\documentclass[a4paper, 11pt]{article}

\usepackage[
            includefoot,  %Uncomment to put page number above margin
            marginparwidth=0in,     % Length of section titles
            marginparsep=0in,       % Space between titles and text
            margin=1.45in,               % 1 inch margins
            includemp]{geometry}

\usepackage{bbm}
\usepackage{float}
\usepackage{graphicx}                  % derni\`ere \'etant la langue principale
\usepackage{amssymb}
\usepackage{amsfonts}
\RequirePackage{amsmath}
\RequirePackage{amsthm}

\RequirePackage{latexsym}
\RequirePackage{amsfonts}
\RequirePackage{amsmath}
\RequirePackage{amssymb}
\RequirePackage{amsthm}

\RequirePackage[backrefs]{amsrefs}

\newtheoremstyle{montheoreme}% name
  {}%				Space above
  {}% 			Space below
  {\itshape}%		Body font
  {}%				Indent amount (empty = no indent, \parindent = para indent)
  {\bf}%			Thm head font
  {.}%			Punctuation after thm head
  {.5em}%			Space after thm head: " " = normal interword space;
  {}%				Thm head spec (can be left empty, meaning `normal')
\newtheoremstyle{maremarque}% name
  {}%				Space above
  {}% 			Space below
  {}%				Body font
  {}%				Indent amount (empty = no indent, \parindent = para indent)
  {\bf}%			Thm head font
  {.}%			Punctuation after thm head
  {.5em}%			Space after thm head: " " = normal interword space;
  {}%				Thm head spec (can be left empty, meaning `normal')
\usepackage{fancyhdr}
\theoremstyle{montheoreme}

\newtheorem{thm}{Theorem}[section]
\newtheorem{defn}[thm]{Definition}
\newtheorem{prop}[thm]{Proposition}
\newtheorem{lem}[thm]{Lemma}

\newtheorem{cor}[thm]{Corollary}
\theoremstyle{maremarque}
\newtheorem{rmq}{Remark}[section]

\DeclareMathOperator{\Ent}{Ent}

\newcommand{\R}{\mathbb{R}}
\newcommand{\N}{\mathbb{N}}
\newcommand{\T}{\mathbb{T}}

\newcommand{\Q}{\mathbb{Q}}

\makeatletter

\@addtoreset{equation}{chapter}
\makeatother

\makeatletter

\@addtoreset{equation}{section}
\makeatother

\begin{document}

\title{A gradient flow approach to large deviations for diffusion processes.}
\author{Max Fathi \thanks{LPMA, University Paris 6, France, max.fathi@etu.upmc.fr.} }
\date{\today}

\maketitle

\begin{abstract}
\bigskip In this work, we investigate links between the formulation of the flow of marginals of reversible diffusion processes as gradient flows in the space of probability measures and path wise large deviation principles for sequences of such processes. An equivalence between the LDP principle and Gamma-convergence for a sequence of functionals appearing in the gradient flow formulation is proved. As an application, we study large deviations from the hydrodynamic limit for two variants of the Ginzburg-Landau model endowed with Kawasaki dynamics.

\end{abstract}

\vspace{1cm}

{\Large \textbf{Introduction}}

In this work, we are interested in the links between the gradient flow formulation of the flow of marginals of stochastic differential equations, and path wise large deviations for sequences of such processes.

Interest in gradient flows on the space of probability measures goes back to [JKO], where it was observed that the heat equation can be viewed as the gradient flow of the entropy 
$$\Ent(\rho) = \int{\rho \log \rho dx}$$
for the Wasserstein distance $W_2$. Note that what we will call here entropy is the negative of the physical entropy. This was later developed into a notion of formal Riemannian structure on $\mathcal{P}(\R^n)$ by Otto in [O]. While a powerful tool to predict the behavior of certain partial differential equations, the point of view of Otto is formal, and we must rely on other tools for proofs. 

Another point of view was developed by Ambrosio, Gigli and Savar\'e in [AGS], which uses the notion of "`minimizing movement"' schemes, developed by De Giorgi and which first appeared in [DGMT], to provide a rigorous framework to define gradient flows on spaces of probability measures. It is based on the idea that gradient flows on $\R^n$ of the form 
$$\dot{x}(t) = -\nabla F(x(t))$$
are the only solutions of 
\begin{equation} \label{eq_grad_flow}
F(x(T)) - F(x(0)) + \frac{1}{2}\int_0^T{|\nabla F(x(t))|^2dt} + \frac{1}{2}\int_0^T{|\dot{x}(t)|^2dt} = 0.
\end{equation}
While the usual gradient flow equation only makes sense in a Riemannian setting (at least in the classical sense), this alternative formulation can be given a meaning in a purely metric setting, as long as we can define a "`length of the gradient"' functional $|\nabla F|$. Section 1.1 concerns this formulation in the setting of the space of probability measures on $\R^n$ endowed with a Wasserstein distance $W_2$, when the functional $F$ is the relative entropy with respect to a nonnegative measure $\mu$, that is
$$
\Ent_{\mu}(\nu) := \int{f \log f d\mu}
$$
if $\nu = f\mu$, and $+\infty$ if $\nu$ is not absolutely continuous with respect to $\mu$. This is the framework developed in the first sections of [AGS].

Several recent papers have been interested in using abstract gradient flow formulations to study convergence of sequences of solutions to partial differential equations. One method, tailored for the case of diffusion processes and based on the discrete approximation of gradient flows, has been devised in [ASZ]. Another, more general, method has been presented in [S]. It consists in studying the asymptotic behavior of sequence of functionals of the form (\ref{eq_grad_flow}) for given functions $F_n$. Informally, it consists in showing that, if the sequence converges in a certain sense to a limiting functional of the same form, with a function $F_{\infty}$, we can directly identify limits of solutions of (\ref{eq_grad_flow}) as gradients flows for the limiting function $F_{\infty}$. 

In the context of statistical physics, the method developed in [S] can be used to prove convergence to the hydrodynamic limit for some models of interacting diffusion processes, such as the Ginzburg-Landau model (see [GPV] or [GOVW] for a presentation of the model, and its hydrodynamic limit). Such results consist in a convergence in probability of the dynamics to a deterministic limiting object, given in general as the solution to some partial differential equation. 

Our aim here is to use the notion of gradient flows to study large deviations from the hydrodynamic limit for interacting spin systems. Such a result consists in proving that the probabilities of a significant deviation from the hydrodynamic limit decays exponentially fast in the system size.  A standard textbook on the topic of large deviations is [DZ], and [KL] contains a review of the literature in the context of large deviations from the hydrodynamic limit for particle systems.

In the recent series of contribution [ADPZ1], [ADPZ2] and [DLR], links between gradient flows in spaces of probability measures and large deviations for many examples of processes arising in statistical physics have been investigated. The main contribution is to show that the gradient flow formulation for partial differential equations such as the heat equation can be deduced from the large deviation principle for $N$ independent stochastic processes given by the stochastic differential equation whose flow of marginals is the solution to the PDE.

In this paper, we prove that process-level large deviations for sequences of diffusion processes are equivalent to the Gamma-convergence of a sequence of functionals that naturally appear in the gradient-flow formulation of these processes. This result generalizes a method used in [DG] and [Fo] to obtain process-level large deviations for the empirical measure of independent Brownian motions. Although these previous works do not discuss gradient flows or optimal transport, there are a lot of similarities between the formalism we use here and their framework, and the proof is based on a similar method.

As an application of this equivalence, we investigate the large deviations for two variants of the Ginzburg-Landau model endowed with Kawasaki dynamics, giving an alternative approach to obtaining the results of [DV] and [Q]. The first model is a random conductance model, and the second one is the non-gradient Ginzburg-Landau model of [Va] and [Q]. As far as the author knows, the large deviation principle for the random conductance model obtained here is new.

\vspace{1cm}

\tableofcontents

\vspace{1cm}

{\Large\textbf{Notations}}

\begin{itemize}

\item If $A$ is a symmetric positive matrix, then $\sqrt{A}$ is the unique symmetric positive matrix whose square is $A$;

\item $W_{2,G}$ is the Wasserstein distance on $P_2(\R^d)$ for the Riemannian distance $d_G$ on $\R^d$ endowed with the metric tensor $G : \R^d \longrightarrow \mathcal{S}_{++}(\R^d)$. It is given by
$$W_{2,G}^2(\mu_0, \mu_1) := \underset{\pi}{\inf} \hspace{1mm} \int{d_G(x,y)^2\pi(dx, dy)}$$
where the infimum is taken over all coupling $\pi$ of the probability measures $\mu_0$ and $\mu_1$;

\item $Z$ is a constant enforcing unit mass for a probability measure;

\item $C$ is a constant that may change from line to line, or even within a line;

\item $C_b(X)$ is the space of real-valued, continuous bounded functions on the space $X$;

\item $\text{div}(A)(x)$ is the vector of $\R^d$ with coordinates $(\text{div}(A)(x))_i := \underset{j = 1}{\stackrel{d}{\sum}} \hspace{1mm} \frac{\partial A_{ij}}{\partial x_j}(x)$, where $A : \R^d \longrightarrow \mathcal{M}_d(\R)$.

\end{itemize}

\vspace{1cm}

\section{Framework and Method}

\subsection{Gradient flows in $P_2(\mathbb{R}^n)$}

In this section, we endow $\mathbb{R}^n$ with a Riemannian structure, given by a metric tensor $G(x)$, and a positive measure $\mu$ that is absolutely continuous with respect to the Lebesgue measure. We also consider the functional on the space of probability measures
\begin{equation}
\Ent_{\mu}(\nu) := \int{f \log f d\mu}
\end{equation}
when $\nu = f\mu$, and that takes value $+\infty$ for probability measures that are not absolutely continuous with respect to $\mu$. Note that although we call this functional the entropy, it is the \emph{negative} of the physical entropy. When $\mu = \exp(-H)$, this functional can be written as
$$\Ent_{\mu}(\nu) = \int{f \log f dx} + \int{H(x) \nu(dx)}.$$

We can endow the space of probability measures with finite second moments with the Wasserstein distance associated to the Riemannian metric structure $W_{2,G}$.

In the sequel, we will consider curves $(\nu_t)_{t \in [0,T]}$ in $\mathcal{P}_2(\R^d)$ which are absolutely continuous, that is there exists a nonnegative function $g \in L^1([0,T])$ (which depends on $(\nu_t)_{t \in [0,T]}$) such that for any $s \leq t$ we have

\begin{equation}
W_{2,G}(\nu_s, \nu_t) \leq \int_s^t{g(r)dr}.
\end{equation}

We will require the following technical assumptions on the metric tensor $G$: 

\begin{equation} \label{assump-tensor1}
\frac{1}{c}|\xi|^2 \leq \langle G(x) \xi, \xi \rangle \leq c|\xi|^2, \forall x \in \R^n, \forall \xi \in \R^n
\end{equation}
for some constant $c$, and
\begin{equation} \label{assump-tensor2}
x \longrightarrow \langle G(x) \xi, \xi \rangle \hspace{5mm} \text{is lower semicontinuous} \hspace{2mm} \forall \xi \in \R^n.
\end{equation}

Since we now have a metric structure on $P_2(\mathbb{R}^n)$, we can define the metric derivative of an absolutely continuous curve $(\nu_t)$ as

\begin{equation}
|\dot{\nu}|(t) := \underset{h \rightarrow 0^+}{\limsup} \hspace{1mm} \frac{1}{h} W_{2,G}(\nu_t, \nu_{t +h}).
\end{equation}

\begin{defn}
We denote by $|| h ||_{\nu}$ the $H^1$ of the norm of the smooth function $h$, defined by
\begin{equation}
|| h ||_{\nu}^2 := \int{\langle A \nabla h, \nabla h \rangle d\nu}
\end{equation}
and $|| \rho ||_{\nu,*}$ its dual norm, given by

\begin{equation} \label{def_dual_norm}
|| \rho ||_{\nu,*}^2 := \underset{h}{\sup} \hspace{1mm} 2\int{h\rho d\nu} - || h ||_{\nu}^2,
\end{equation}
where the supremum runs over all smooth function $h : \R^n \rightarrow \R$.
\end{defn}

Let $g$ be defined by

\begin{equation} \label{def_upper_gradient}
g(\nu) := \left(\int{ \left\langle \frac{A\nabla \nu}{\nu} + A\nabla H, \frac{\nabla \nu}{\nu} + \nabla H \right\rangle d\nu}\right)^{1/2}
\end{equation}
if $\nu$ is absolutely continuous with respect to the Lebesgue measure, and $+\infty$ else wise.

The following result explains how the functional $g$ can be used to control the variation in relative entropy for absolutely continuous curves in the space of probability measures. Its proof in this context can be found in [L] (which generalizes previous results of [AGS]).

\begin{prop}
$g$ is an upper gradient for $\Ent_{\mu}$, i.e. for every absolutely continuous curve $(\nu_t)_{0 \leq t \leq T}$ we have 
$$|\Ent_{\mu}(\nu_t) - \Ent_{\mu}(\nu_s)| \leq \int_s^t{g(\nu_r)|\dot{\nu}|(r)dr}$$
for every $0 \leq s \leq t \leq T$.
\end{prop}

\begin{defn}
Let $(\nu_t)_{t \in [0,T]}$ be a time-dependent family of measures that is absolutely continuous.We say it is a gradient flow of the functional $\Ent_{\mu}$ if 
\begin{equation} \label{def_gf}
\frac{d}{dt} H_{\mu}(\nu_t) = -\frac{1}{2}g(\nu_t)^2 - \frac{1}{2}|\dot{\nu}|(t)^2
\end{equation}
for almost every $t \in [0,T]$.
\end{defn}

Gradient flows for the Wasserstein structure on $\mathcal{P}(\R^d)$ endowed with an Euclidean structure have been studied in [AGS], and it turns out they are related to the heat equation. Their results were then generalized to the case of a Riemannian structure in [L]: 

\begin{prop}
$(\nu_t)_{t \in [0,T]}$ satisfies (\ref{def_gf}) iff the densities $f(t,\cdot) = \frac{d\nu_t}{d\mu}$ form a weak solution of the parabolic PDE
\begin{equation}
\frac{\partial f\mu}{\partial t} = \nabla \cdot(A (\nabla f)\mu), 
\end{equation}
where $A = G^{-1}$.
\end{prop}

As a consequence of this result and Ito's formula, we also have a representation of gradient flows as the flow of laws of the solution to a SDE:

\begin{prop} If $(\nu_t)_t$ is a gradient flow of $\Ent_{\mu}$ with $\mu = \exp(-H)dx$, then it is the flow of marginals of a solution of the SDE
\begin{equation} \label{sde}
dX_t = -A(X_t)\nabla H(X_t)dt + \text{div}(A)(X_t) \sqrt{2A} \hspace{1mm} dB_t
\end{equation}
with initial condition $X_0$ that has law $\nu_0$. 

\end{prop}

\begin{rmq} 
Diffusion processes that can be written in the form (\ref{sde}) are necessarily reversible.
\end{rmq}

We shall now give the definition of a key functional, which allows us to characterize gradient flows:

\begin{prop}
Let 
\begin{equation} \label{def_j}
J((\nu_t)_t) := H_{\mu}(\nu_0) - H_{\mu}(\nu_T) - \frac{1}{2}\int_0^T{g(\nu_t)^2 + |\dot{\nu}|(t)^2 dt}.
\end{equation}
Then $(\nu_t)_{t \in [0,T]}$ is a gradient flow of the functional $H_{\mu}$ iff $J((\nu_t)_t) = 0$.
\end{prop}

\begin{rmq}
In this setting, we have the following alternate formulation for the functional $J$:

\begin{equation}
J((\nu_t)_t) = \frac{1}{2}\int_0^T{||\dot{\nu} - \nabla \cdot(A (\nabla f)\mu) ||_{\nu, *}^2dt},
\end{equation}
at least for smooth functions. This formulation may seem more convenient, but in this context, (\ref{def_j}) will be easier to manipulate. 
\end{rmq}

\begin{rmq}
We can make an interpretation of the notion of gradient flows in a statistical physics framework. It is a well-known principle in equilibrium statistical physics that steady states can be identified as minimizers of a thermodynamic functional, such as free energy, as a consequence of the second principle of thermodynamics. Here it is the relative entropy $\Ent_{\mu}$ which plays the role of free energy, and indeed it minimizer is the equilibrium state $\mu$. The gradient flow formulation identifies the correct trajectory as the minimizer of some action functional. This can be seen as an extension of the minimization principle to non-equilibrium statistical physics, with correct trajectories being those that decrease the free energy as fast as possible.
\end{rmq}

\subsection{Relative entropy and large deviations}

In this section, we introduce the notions of relative entropy and large deviations, and the links between the two. 

\begin{defn} [Relative entropy]
Given two probability measures $P$ and $Q$ on a Polish space $X$, the relative entropy of $P$ with respect to $Q$ is given by
$$H(P;Q) := \underset{f \in C_b(X)}{\sup} \hspace{1mm} \mathbb{E}_P(f) - \log \mathbb{E}_Q(e^f).$$
\end{defn}

The following result is well-known, and is obtained by a computation of the Legendre transform (see Lemma 6.2.13 in  [DZ]).

\begin{prop}
We have 
$$H(P;Q) = \mathbb{E}_P \left[ \ln \left(\frac{dP}{dQ}\right) \right]$$
if $P$ is absolutely continuous with respect to $Q$, and $H(P;Q) = +\infty$ if not.
\end{prop}

We now define large deviations: 

\begin{defn}
Let $I$ be a lower semicontinuous, nonnegative function on a Polish space $X$ and $(a_n)_{n \in \N}$ a sequence of increasing, positive real numbers that goes to infinity. A sequence of probability measures $\mathbb{P}_n$ on $X$ is said to satisfy a large deviation principle with speed $(a_n)_n$ and good rate function $I$ iff

(i) For any closed set $F$, $\limsup a_n^{-1}\log \mathbb{P}_n(F) \leq -\underset{x \in F}{\inf} \hspace{1mm} I(x);$

(ii) For any open set $O$, $\liminf a_n^{-1}\log \mathbb{P}_n(O) \geq -\underset{x \in O}{\inf} \hspace{1mm} I(x).$
\end{defn}

Informally, this definition means that $\mathbb{P}_n(X_n \approx x) \approx \exp(-a_n I(x))$. We refer to the textbook [DZ] for an introduction to large deviations.

In the recent contribution [Ma], Mariani proved the equivalence between large deviation principles for sequences of probability measures and the Gamma convergence (which we define below) of the associated relative entropy functionals.

\begin{defn}[Gamma convergence]
Let $X$ be a space endowed with a notion of convergence. A sequence $(I_n)$ of functionals on $X$ is said to $\Gamma$-converge to a functional $I$ at point $x \in X$ if the two following conditions are met: 

(i) For any sequence $(x_n)$ that converges to $x$, we have $\underset{n \rightarrow \infty}{\liminf } I_n(x_n) \geq I(x)$;

(ii) There exists a sequence $(x_n)$ that converges to $x$ such that $\underset{n \rightarrow \infty}{\lim } I_n(x_n) = I(x)$.

The sequence $(I_n)$ is said to $\Gamma$-converge to $I$ if it $\Gamma$-converges to $I$ at every point.
\end{defn}

Before we give a statement of Mariani's result, we define the notion of exponential tightness : 

\begin{defn}[Exponential tightness]
A sequence of probability measures $(\mu_n)$ on a topological space $X$ is said to be exponentially tight with speed $(a_n)$ if, for any $\alpha > 0$, there exists a compact set $K_{\alpha}$ such that 
$$\underset{n}{\limsup} \hspace{1mm} \frac{1}{a_n} \log \mu_n(K_{\alpha}^c) \leq -\alpha.$$
\end{defn}

\noindent Mariani's result can be stated as follows : 

\begin{thm}[Ma, 2012] \label{thm_mariani}
Let $(\mu_n)$ be a sequence of probability measures on a Polish space $X$,  $(a_n)$ a sequence of positive real numbers such that $\underset{n}{\lim} \hspace{1mm} a_n = +\infty$ and $I : X \rightarrow [0,+\infty]$ a measurable, lower semicontinuous functional. We endow the space of probability measures with the topology of weak convergence. 

(i) The sequence $(\mu_n)$ satisfies a large deviations upper bound with speed $(a_n)$ and rate function $I$ iff it is exponentially tight with speed $(a_n)_n$ and if for any sequence $(\nu_n)$ of probability measures on $X$ that weakly converges to a Dirac measure $\delta_x$, we have
$$\liminf \frac{1}{a_n}H(\nu_n, \mu_n) \geq I(x);$$

(ii) The sequence $(\mu_n)$ satisfies a large deviations upper bound with speed $(a_n)$ and rate function $I$ iff for any point $x$, there exists a sequence $(\nu_n)$ of probability measures on $X$ that weakly converges to $\delta_x$ such that
$$\underset {n \longrightarrow \infty}{\limsup} \hspace{1mm} \frac{1}{a_n}H(\nu_n, \mu_n) \leq I(x).$$

\end{thm}

A heuristic explanation of Theorem \ref{thm_mariani} can be made in terms of the Bryc-Varadhan theorem (see sections 4.3 and 4.4 in [DZ]). One can relate the relative entropy functional and exponential moments of functions by the relation
$$\int{\exp(f)d\mu} = \underset{\nu \in \mathcal{P}(X)}{\sup} \int{fd\nu} - \Ent_{\mu}(\nu).$$
Since the Bryc-Varadhan lemma states that we can understand the large deviations for sequences of measures by looking at $\frac{1}{a_n}\log \int{\exp(a_n f)d\mu_n}$ for bounded continuous functions, the above relation translates the problem to investigating the behavior of the sequence of relative entropy functionals.

\begin{rmq} \label{rq1} For the lower bound, it is enough to check the existence of a recovery sequence for every point $y$ in a subset $Y$ of $X$, such that given $x \in X$, there exists a sequence $(y_k)$ of elements of $Y$ that converges to $x$, and such that $I(y_k)$ converges to $I(x)$.
\end{rmq}

\subsection{Relative entropy for the law of processes}

Our first result is a relation between relative entropy with respect to the law of a solution of (\ref{sde}) and the functional $J$ associated to the gradient flow formulation  of the flow of marginals. It is a generalization of Theorem (1.31) in part II of [Fo] (which dealt with the case of independent Brownian motions).

\begin{thm} \label{min_ent} 
Let $Q$ be the law of a solution to a SDE of type (\ref{sde}) on a space $\mathbb{R}^{d}$, and $P$ the law of a process with finite relative entropy with respect to $Q$, with flow of marginals $(\nu_t)$. Then:

(i) We have the lower bound 

$$H(P;Q) \geq H(P_0,Q_0) + \frac{1}{2}J((\nu_t))$$
where $P_0$ and $Q_0$ are the laws of the initial conditions.

(ii) There exists a process with law $\tilde{P}$ that has the same flow of marginals as $P$, such that 

$$H(\tilde{P}, Q) = H(P_0,Q_0) + \frac{1}{2}J((\nu_t)).$$
\end{thm}

As a direct consequence of this relation and Theorem \ref{thm_mariani}, we can use the functionals $J$ to study large deviations.

We consider a sequence of diffusion processes of the form (\ref{sde}). The parameters we allow to vary are the drift $\nabla H$, the diffusion coefficient $A$, and the dimension of the underlying space $d_n$.

To be able to state a large deviation principle for the laws of these diffusion processes, we need to embed their trajectories into a single space. We therefore implicitly assume that all the spaces $\R^{d_n}$ have been embedded into a single metric space $X$. We then endow the space $\mathcal{C}([0,T],X)$ with some topology that makes it a metric, separable space. A typical choice would be the supremum norm.

\begin{cor} \label{cor_ent}
Let $Q_n$ be the law of a stochastic differential equation of the form $\ref{sde}$, with $Q_{0,n}$ the law of the initial condition, and let $J_n$ be the functional involved in the gradient flow formulation of the flow of marginals of Proposition \ref{def_j}.
Then for any continuous trajectory $x \longrightarrow x_t$, the sequence of normalized relative entropy functionals $a_n^{-1}H(\cdot, Q_n)$ $\Gamma$-converge at point $\delta_{x}$ iff the functionals $\frac{1}{a_n}(H(\cdot,Q_{0,n}) + \frac{1}{2}J_n(\cdot))$ also do, with the same $\Gamma$-limit.

As a consequence, solutions of gradients flows satisfy a large deviations principle with speed $(a_n)$ and rate function $I$ iff the functionals $\frac{1}{a_n}(H(\cdot,Q_{0,n}) + \frac{1}{2}J_n(\cdot))$ $\Gamma$-converge to $I$ at the Dirac measure $\delta_{(x_t)}$, for every continuous trajectory $t \rightarrow x_t$, and if the sequence of laws is exponentially tight.

\end{cor}

Of course, studying relative entropy to understand large deviations is a known technique (see [DG] and [Fo]). Our contribution is to show that instead of studying relative entropy (which depends on the law of the whole trajectory), we can study the functional $J$, which only depends on the flow of marginals, and is easier to manipulate, at least in some cases of interest, due to its connexion with optimal transport. It should be noted that some of the ideas we use here (relative entropy, variational formulations for rate functions) are reminiscent of those used in [DG] to study large deviations for weakly-interacting mean-field models.

\begin{rmq}
We expect the rate function $I$ to be of the form $I(x_t) = I_0(x_0) + J(x_t)$, with $J$ the function involved in the formulation (\ref{eq_grad_flow}) of a gradient flow in a certain metric space. This comes from the fact that often the rate function $I$ will have a unique minimizer, which will be the deterministic limit of our sequence of processes. So we can reformulate the "`correct"' limit as the unique minimizer of a function. But gradient flow formulations also characterize some "`correct"' path as the unique minimizer of a functional. The similarity between these two point of views make us expect that they will be related, and this turns out to often be the case. See [ADPZ1] and [ADPZ2] for a study of this link in the case of sequences of independent processes.
\end{rmq}

\begin{rmq}
An important element of the study of the functional $J$ is the study of the Fischer information, or entropy-production functional. The importance of this functional can be understood in terms of statistical physics. It is a well-known principle in equilibrium statistical physics that equilibrium states can be obtained by optimizing some thermodynamic quantity, such as the free energy. This principle sometimes carries through to \emph{non-equilibrium} statistical physics. Since the system seeks to increase the physical entropy (and therefore decrease the mathematical entropy), we can look at the entropy production functional, which we seek to optimize. Gamma-convergence corresponds to convergence of minimizers, so we can expect the "`correct"' trajectories to be those that, in the limit, make the entropy production functional as small as possible.
\end{rmq}

\subsection{Some questions}

\begin{itemize}

\item Is there a similar phenomenon for the large deviations of discrete systems, such as interacting particle systems? In the recent paper [M], Maas showed that any reversible Markov chain on a finite space can be written as a gradient flow of a relative entropy for a well-chosen Riemannian structure on the space of probability measures. Can we exploit this structure to get the large deviations of systems such as a zero-range process, or exclusion processes?

\item Many partial differential equations can be written as gradient flows for energy functionals which are not the relative entropy, such as porous medium equations (see [O]). The energy production functional $J$ still characterizes such gradient flows. Are there any nice properties implied by Gamma-convergence of the functional $J$ for such systems?

\item We apply in Section 3 this principle to get the large deviations for a system of diffusions with nearest neighbor interaction. It would also be interesting to look at mean-field models, where each diffusion interacts with all the others. In the case of smooth mean-field interactions, the question has been solved in [DG], with a method that is very similar to the one we use here. A natural question is whether this extends to singular interactions. A case of interest is that of Coulomb interactions.

\item Another natural question is whether we can use this principle in a context of modelization. Say we wish to approximate a phenomenon characterized as the unique solution to a partial differential equation of the form $\partial_t \rho = H(\rho, \nabla \rho, ..)$ with a system of $N$ interacting diffusion, with $N$ large. If we can find a sequence of diffusion processes on $\R^N$  and a sequence of positive numbers $a_N$ such that $\frac{1}{a_N}J_N$ Gamma-converges to a lower semicontinuous functional that has the solution to the PDE as sole minimizer, then these diffusion processes form a good approximation. Can this idea be exploited in this context? This would be particularly interesting if we can extend our results to sequences of interacting particle systems.

\item Our method works for reversible diffusion processes. Is their a similar method that works for non reversible processes, such as interacting diffusion processes with a boundary condition, or second-order diffusion processes? A method has been recently developed in [DPZ] in what is called the GENERIC framework.

\end{itemize}

\section{Proof of Theorem \ref{min_ent}}

The proof is a generalization of ideas coming from Part II of [Fo]. It will consist in three steps : first we shall use Girsanov's theorem to give a representation of the law of processes that are absolutely continuous with respect to the law of the gradient flow. In a second step, we shall give a representation of the relative entropy of such a process, and finally we shall use this representation to obtain the lower bound of our theorem.

The following result is a direct application of Girsanov's theorem (see for example [Le, Theorem 2.4]):
\begin{prop}
Let $Q$ be the law of the solution of (\ref{sde}) on $[0,T]$, and $P$ the law of a process that is absolutely continuous with respect to $Q$. Then there exists an adapted process $b_t$ valued in $\mathbb{R}^d$ such that 
\begin{equation} \label{density_process}
\frac{dP}{dQ}((X_t)_{0\leq t \leq T}) = 1_{\frac{dP}{dQ} > 0} \frac{dP_0}{dQ_0}(X_0) \exp \left(\int_0^T{b_t \cdot \sqrt{2A(X_t)}dB^P_t} + \int_0^T{\langle A(X_t)b_t, b_t \rangle dt}\right).
\end{equation}

In this equation, $B^P$ is a P-Brownian motion, that is a local martingale under $P$ which $P$-almost surely has quadratic variation equal to $t$. Moreover, $P$ can be viewed as the law of a solution to the SDE
\begin{equation} \label{gen_sde}
dX_t = A(X_t)(2b_t - \nabla H(X_t))dt + \text{div}(A)(X_t)dt + \sqrt{2A(X_t)}dB_t^P
\end{equation}

As a consequence, the relative entropy is given by 

\begin{equation}
H(P; Q) = H(P_0,Q_0) +  \mathbb{E}_P \left[\int_0^T{\langle A(X_t)b_t, b_t \rangle dt} \right].
\end{equation}
\end{prop}

Note that the relative entropy doesn't only depend on the flow of marginals $P_t$, but on the law of the whole trajectory, unlike the functional $J$.

\begin{lem}[Markov version of the process]
Let $b_t = b(t,(X_s)_{0 \leq s \leq t})$ be the adapted process associated to a law $P$. Define 
\begin{equation} \label{mark_ver}
\tilde{b}_t(x) := \mathbb{E}_P[ b(t,(X_s)_{0 \leq s \leq t}) | X_t = x ].
\end{equation}

Then the process defined by 
\begin{equation} \label{mark_sde}
dX_t = A(X_t)(2\tilde{b}_t(X_t) - \nabla H(X_t))dt + \text{div}(A)(X_t)dt + \sqrt{2A(X_t)}dB_t
\end{equation}
with initial condition $X_0 \sim P_0$ is a Markov process, and its law has the same flow of marginals as $P$.
\end{lem}

\begin{proof}
The fact that this process is a Markov process is a classic result on SDEs, so we shall concentrate on proving that both processes have same marginals. Let $g$ be a smooth function, $X_t$ be a solution of (\ref{gen_sde}) and $\tilde{X}_t$ a solution of (\ref{mark_sde}). We have
\begin{align}
\mathbb{E}[g(\tilde{X}_t)] &= \int_0^t{\mathbb{E}\left[A(\tilde{X}_s)\nabla g(\tilde{X}_s) \cdot (2\tilde{b}_s(X_s) - \nabla H(X_s))\right]ds} \notag \\
& \hspace{1cm} + \int_0^t{\mathbb{E}\left[\text{div}(A(\tilde{X}_s)\nabla g(\tilde{X}_s))\right]ds} \notag \\
&= \int_0^t{\mathbb{E}\left[A(\tilde{X}_s)\nabla g(\tilde{X}_s) \cdot (2\mathbb{E}_P[ b(s,(X_r)_{0 \leq r \leq s}) | X_s] - \nabla H(X_s))\right]ds} \notag \\
& \hspace{1cm} + \int_0^t{\mathbb{E}\left[\text{div}(A(\tilde{X}_s)\nabla g(\tilde{X}_s))\right]ds} \notag 
\end{align}
and
\begin{align}
\mathbb{E}[g(X_t)] &= \int_0^t{\mathbb{E}\left[A(X_s)\nabla g(X_s) \cdot (2b(s,(X_r)_{0 \leq r \leq s}) - \nabla H(X_s))\right]ds} \notag \\
& \hspace{1cm} + \int_0^t{\mathbb{E}\left[\text{div}(A(X_s)\nabla g(X_s))\right]ds} \notag \\
&= \int_0^t{\mathbb{E}\left[A(X_s)\nabla g(X_s) \cdot (2\mathbb{E}_P[b(s,(X_r)_{0 \leq r \leq s})|X_s] - \nabla H(X_s))\right]ds} \notag \\
& \hspace{1cm} + \int_0^t{\mathbb{E}\left[\text{div}(A(X_s)\nabla g(X_s))\right]ds} \notag \\
&= \int_0^t{\mathbb{E}\left[A(X_s)\nabla g(X_s) \cdot (2\tilde{b}(X_s) - \nabla H(X_s))\right]ds} \notag \\
&\hspace{5mm} + \int_0^t{\mathbb{E}\left[\text{div}(A(X_s)\nabla g(X_s))\right]ds}. \notag
\end{align}

This shows that the marginals of the laws of $X$ and $\tilde{X}$ satisfy the same parabolic PDE
$$\frac{\partial f}{\partial t} = \text{div}(A\nabla f) + \text{div}(A(2\tilde{b} - \nabla H)f).$$ 
Since they have the same initial condition, and since solutions to such PDEs are unique, they are the same.
\end{proof}

\begin{lem}
Let $\tilde{P}$ be the law of the solution of (\ref{mark_ver}).

We have 
\begin{equation}
H(P;Q) \geq H(\tilde{P};Q)
\end{equation}
\end{lem}

\begin{proof}
We already know that 
\begin{equation} \label{ent_mark_version}
H(\tilde{P}; Q) = H(P_0; Q_0) + \mathbb{E}_{\tilde{P}}\left[\int_0^T{\langle A(X_t)\tilde{b}_t, \tilde{b}_t \rangle dt} \right].
\end{equation}

An application of Jensen's inequality and the definition of $\tilde{b}$ yields
\begin{align}
H(P;Q) - H(P_0,Q_0) &= \mathbb{E}_P\left[\int_0^T{\langle A(X_t)b_t, b_t \rangle dt} \right] \notag \\
&= \int_0^T{\mathbb{E}_P\left[\langle A(X_t)b_t, b_t \rangle \right] dt} \notag \\
&= \int_0^T{\mathbb{E}_{P_t} \left[\mathbb{E}_P\left[\langle A(X_t)b_t, b_t \rangle \middle| X_t  \right]\right] dt} \notag \\
&\geq \int_0^T{\mathbb{E}_{P_t} \left[\langle A(X_t) \tilde{b}(X_t), \tilde{b}(X_t) \rangle \right] dt} \notag \\
&= H(\tilde{P}, Q) - H(P_0,Q_0),
\end{align}
which is the desired lower bound.
\end{proof}

\begin{lem}[Entropy of the Markov process]  The entropy of the Markov version of the process satisfies 
$$H(\tilde{P};Q) = H(\nu_0,Q_0) + \frac{1}{2}J((\nu_t)_t)$$

where $\nu_t$ is the flow of marginals of the process $P$.
\end{lem}

\begin{proof}
Let $g$ be a smooth, compactly supported function. It\^o's formula applied to the SDE (\ref{mark_sde}) yields
\begin{align} \label{ito}
\mathbb{E}[g(X_t)] = \mathbb{E}[g(X_0)] &+ \int_0^t{\mathbb{E}[A\nabla g(X_s) \cdot (2\tilde{b}_s(X_s) - \nabla H(X_s))] ds} \notag \\
&+ \int_0^t{\mathbb{E}[(\nabla \cdot A(X_s) \nabla g(X_s))]ds}
\end{align}

It is easy to deduce from the Ito formulation (\ref{ito}) that the flow of marginals $\nu_t$ solves (in a weak sense) the PDE

\begin{align}
\dot{\nu}_t &= -\text{div}\left(2A\tilde{b}_t\nu_t - \nu_tA\nabla H - A\nabla \nu_t \right). \notag \\
\end{align}

Therefore the variation of the entropy of the marginals is given by
\begin{align} \label{dif_ent}
\Ent_{\mu}(\nu_T) &- \Ent_{\mu}(\nu_0) = \int_0^T{\int{A\nabla \nu_t(x) \cdot (2\tilde{b}_t(x) - \nabla H(x))dx}dt} \notag \\
& \hspace{5mm} - \int_0^T{\int{\frac{\langle A(x)\nabla \nu_t(x), \nabla \nu_t(x)\rangle}{\nu_t(x)} dx}dt} \notag \\
& \hspace{5mm} + \int_0^T{\int{A(2\tilde{b}_t - \nabla H) \cdot \nabla H d\nu_t}dt} - \int_0^T{\int{A\nabla H \cdot \nabla \nu_t dx}dt} \notag \\
& = \int_0^T{\int{2A\tilde{b}_t \cdot \nabla \nu_t + \nu_t\nabla H dx}dt} - \int{\frac{A(\nabla \nu_t + \nu_t \nabla H)\cdot (\nabla \nu_t + \nu_t \nabla H)}{\nu_t} dx}
\end{align}

From the Benamou-Brenier formula for $W_2$ (see for example [Vi1]):

$$W_{2,G}^2(\nu_0, \nu_1) = \inf \left\{ \int_0^1{\int{\langle G v, v \rangle d\nu_t }dt}; \hspace{3mm} \dot{\nu} + \text{div}(v\nu_t) = 0 \right\}$$

we deduce 

\begin{align} \label{prod_metric_der}
\frac{1}{2}&\int_0^T{|\dot{\nu}_t|^2dt} \notag \\
&= \frac{1}{2}\int_0^T{\left\langle A\left(2\tilde{b}_t - \nabla H - \frac{\nabla \nu_t}{\nu_t}\right), \left(2\tilde{b}_t - \nabla H + \frac{\nabla \nu_t}{\nu_t}\right) \right\rangle d\nu_t dt} \notag \\
&= 2\int_0^T{\int{ \langle A\tilde{b}_t, \tilde{b}_t \rangle d\nu_t}dt} + \frac{1}{2}\int_0^T{\int{\left\langle A\left(\nabla H + \frac{\nabla \nu_t}{\nu_t}\right), \left(\nabla H + \frac{\nabla \nu_t}{\nu_t}\right) \right\rangle d\nu_t} dt} \notag \\
& \hspace{5mm} - \int_0^T{\int{ \left\langle 2A\tilde{b}_t, \left(\nabla H + \frac{\nabla \nu_t}{\nu_t}\right) \right\rangle d\nu_t} dt}.
\end{align}

By the definition (\ref{def_upper_gradient}) of the upper gradient $g$, we have

\begin{equation} \label{prod_upper_grad}
\frac{1}{2}\int_0^T{g(\nu_t)^2dt} = \frac{1}{2}\int_0^T{\int{ \left\langle \frac{A\nabla \nu_t}{\nu_t} + A\nabla H, \frac{\nabla \nu_t}{\nu_t} + \nabla H \right\rangle d\nu_t}dt}.
\end{equation}

Combining (\ref{dif_ent}), (\ref{prod_metric_der}) and (\ref{prod_upper_grad}), we get 

\begin{equation} 
J((\nu_t)_t) =2\int_0^T{\int{ \langle A\tilde{b}_t, \tilde{b}_t \rangle d\nu_t}dt},
\end{equation}

and then the lemma immediately follows from (\ref{ent_mark_version}).
\end{proof}

To deduce Corollary \ref{cor_ent} from Theorem \ref{min_ent}, the only thing we still have to prove is that, if $(\nu_t)$ is a flow of marginals such that $J((\nu_t))$ is finite, there exists a process whose law is absolutely continuous with respect to $Q$, and with flow of marginals $(\nu_t)$.

Let $(\nu_t)$ be an absolutely continuous flow of marginals such that $J((\nu_t))$ is finite. From [L, Theorem 2.4], we know that there exists a vector field $(v_t)$ such that the continuity equation
\begin{equation}
\dot{\nu}_t = \nabla \cdot (Av_t \nu_t)
\end{equation}
is satisfied. On the other hand, we know that $(\nu_t)$ is the flow of marginals of the solution to an SDE of type 
$$dX_t = 2A(X_t)b_t(X_t)dt + \sqrt{2A(X_t)}dB_t,$$
whose law would then be absolutely continuous with respect to $Q$, if the flow solves in a weak sense the PDE
$$\dot{\nu}_t = \text{div}(A(\nabla \nu_t + 2b_t\nu_t)).$$
Since $J((\nu_t))$ is finite, the upper gradient $g(\nu_t)$ is finite for almost every $t$, and $\nabla \nu_t$ exists. We therefore only have to take $2b_t(x) = v_t(x) - \frac{\nabla \nu_t}{\nu_t}(x)$ to see that the flow solves the above PDE.

\section{Large deviations for the Ginzburg-Landau model}

\subsection{The model}

The (classical) Ginzburg-Landau model equipped with Kawasaki dynamics is a sequence of $N$ diffusions, interacting according to the SDE

$$dX^i_t = N^2(\psi(X^{i+1}_t) + \psi(X^{i-1}_t) - 2\psi(X^i_t))dt + \sqrt{2}N(dB^{i+1}_t - dB^i_t).$$

This diffusion is not ergodic on the whole space $\R^N$, since it preserves the quantity $\sum Xî_t$, but it is ergodic when restricted to a hyperplane 
\begin{equation}
X_{N,m} := \left\{ x \in \R^N; \hspace{3mm} \frac{1}{N} \underset{i = 1}{\stackrel{N}{\sum}} \hspace{1mm} x_i = m \right\}.
\end{equation}
It then has an invariant measure 
\begin{equation}
\mu(dx) := \frac{1}{Z}\exp\left(-\frac{1}{N} \underset{i = 1}{\stackrel{N}{\sum}} \hspace{1mm} \psi(x_i)\right) \mathbbm{1}_{x \in X_{N,m}}\mathcal{L}^{N-1}(dx)
\end{equation}
where $\mathcal{L}$ is the Lebesgue measure on $X_{N,m}$. Equivalently, its law is given by the solution of the PDE

\begin{equation}\label{gl_eq}
\frac{\partial f\mu}{\partial t} = \text{div}(A_0\nabla f \mu),
\end{equation}
where $f$ is the density with respect to $\mu$, and $A_0$ is the discrete Laplacian scaled by $N^2$, that is 
$$(A_0){i,j} := N^2(\delta_{i, j+1} + \delta_{i, j-1} - 2\delta_{i, j}).$$

When $N$ goes to infinity, if the initial condition behaves deterministically in the limit, the (properly rescaled) solutions concentrate around a deterministic profile, called the hydrodynamic limit, which has been studied in [GPV]. Large deviations from this hydrodynamic limit have been studied in [DV]. 

We shall investigate the large deviations as $N$ goes to infinity for two versions of this model, with conductances. The first case will involve random conductances, and the second will involve conductances depending on the configuration of spins. Our method will rely on Corollary \ref{cor_ent}, and reduce the problem to the study of the behavior of the functional $J_N$ associated with the gradient flow formulation of these dynamics.

For technical reasons, we shall assume that the initial data follows a local Gibbs state, that is 
\begin{equation} \label{init_GL}
f_0(x) = \frac{1}{Z}\exp\left(\sum x_i \varphi'(\rho_0(i/N)) \right)
\end{equation}
for some continuous function $\rho_0$. We will later see that this initial data concentrates around the deterministic profile $\rho_0$. It can be shown that, for initial data that behaves deterministically in the limit, solutions at any positive time are close (in the sense of relative entropy) to such a local Gibbs state. See [K] or [F1] for a proof.

We will also assume that the single-site potential $\psi$ is of the form 
\begin{equation} \label{hyp_psi}
\psi(x) = \frac{1}{p}x^p + \delta\psi(x)
\end{equation}
for some $p \geq 0$ and a perturbation $\delta\psi$ that is $C^2$, bounded and with bounded first and second derivative.

It is likely that the results hold for more general functions $\psi$, but such a result would require more general technical tools than those developed in the next section. For example, if $\psi$ doesn't grow at least as $|x|^2$ as $x$ goes to infinity, then the logarithmic Sobolev inequality doesn't hold. In [DV], the LDP is proved for the case where $\psi$ is only superquadractic, and $\psi' = o(\psi)$.

It turns out that the proof of the upper bound in the convergence of $J_N$ is the same as the proof of [Q] of the lower bound for the LDP. We shall therefore only sketch the proofs of the upper bounds, and concentrate on the lower bounds in the $\Gamma$-convergence.

\subsection{Some technical estimates}

In this section, we give a few technical results, collected from various sources, which we shall use in the proofs of the large deviation principles. Most of them are classical results in the study of hydrodynamic limits, and we will often only give a brief sketch of the proofs, or simply refer to the original source.

We will use logarithmic Sobolev inequalities, which we now define: 

\begin{defn}
Let $X$ be a Riemannian manifold. A probability measure $\mu$ on $X$ is said to satisfy a LSI with constant $\rho > 0$ if, for any locally Lipschitz, nonnegative function $f \in \textsl{L}^1(\mu)$,
$$\int{f \log (f) d\mu} - \left(\int{f d\mu}\right)\log \left(\int{f d\mu}\right) \leq \frac{1}{\rho}\int{\frac{|\nabla f|^2}{2f}d\mu}.$$
\end{defn}

The following result was proven in [MO]: 

\begin{thm}
Under the assumption (\ref{hyp_psi}), the measures $\mu_N$ satisfy the logarithmic Sobolev inequality
$$\Ent_{\mu_N}(g) \leq C\int{\frac{|\nabla g|^2}{g} d\mu_N}$$
for any nonnegative, locally Lipschitz function $g$, with constant $C$ independent of the dimension $N$ and the mean spin $m$. Combined with the discrete Poincar\'e inequality, this implies
$$\Ent_{\mu_N}(g) \leq C\int{\frac{\langle A_0 \nabla g, \nabla g \rangle}{g} d\mu_N}$$
for some constant $C$ that is independent of the dimension and the mean spin.
\end{thm}

As a consequence of this result and of [GOVW, Lemma 26], we have the following result: 

\begin{lem} \label{lem_borne_l2}
Let $f_N$ be a sequence of probability densities with respect to $\mu_N$ such that 
$$\underset{N}{\sup} \hspace{1mm} \frac{1}{N}\int{\frac{\langle A_0 \nabla f, \nabla f \rangle}{f} d\mu_N} < +\infty.$$
Then we also have
$$\underset{N}{\sup} \hspace{1mm} \frac{1}{N}\int{\sum |x_i|^2 f(x) \mu_N(dx)} < +\infty.$$
This result still holds if we replace $\mu_N$ by another sequence of measures with bounded second moment and which satisfy a LSI with uniform constant.
\end{lem}

This can be generalized to the following result: 

\begin{lem} \label{lem_borne_lp}
Assume that $\psi$ is of the form $\frac{1}{p}|x|^p + \delta\psi(x)$. Let $f_N$ be a sequence of probability densities with respect to $\mu_N$ such that 
$$\underset{N}{\sup} \hspace{1mm} \frac{1}{N}\int{\frac{\langle A_0 \nabla f, \nabla f \rangle}{f} d\mu_N} < +\infty.$$
Then we also have
$$\underset{N}{\sup} \hspace{1mm} \frac{1}{N}\int{\sum |x_i|^p f(x) \mu_N(dx)} < +\infty.$$
\end{lem}

\begin{proof}
It has been shown in [F2] that, under our assumptions on $\psi$, $\mu_N$ satisfies the following transport-entropy inequality: for any probability measure $\nu_N$, 
$$W_p^p(\nu_N, \mu_N) \leq C\Ent_{\mu_N}(\nu_N)$$
for some constant $C > 0$ that does not depend on $N$, and $W_p$ is the $L^p$ Wasserstein distance
$$W_p^p(\nu, \mu) := \underset{\pi}{\inf} \hspace{1mm} \int{\sum |x_i - y_i|^p \pi(dx, dy)}.$$ 
From the $W_p$-Lipschitz continuity of p-moments (see [Vi2, Proposition 7.29]), we know that
$$\left(\int{\sum |x_i|^p\nu_N(dx)}\right)^{1/p} - \left(\int{\sum |x_i|^p\mu_N(dx)}\right)^{1/p} \leq W_p(\mu_N, \nu_N)$$
so that
$$\int{\sum |x_i|^p\nu_N(dx)} \leq C\Ent_{\mu_N}(\nu_N) + C\int{\sum |x_i|^p\mu_N(dx)}$$
Since $\mu_N$ satisfies a logarithmic Sobolev inequality, $\Ent_{\mu_N}(\nu_N) \leq CN$, and it is also easy to see that 
$$\int{\sum |x_i|^p\mu_N(dx)} \leq CN,$$
which concludes the proof.
\end{proof}

We now give a version the version of the local Cram\'er theorem we shall use:

\begin{thm} \label{lct}
Let $(a_i)$ be some sequence of real numbers. We define
$$\psi_K(m) := -\frac{1}{K}\log \int_{X_{K,m}}{\exp(\sum a_i x_i + \psi(x_i))dx}$$
and 
$$\varphi_{K}(m) := \underset{\sigma \in \mathbb{R}}{\sup} \hspace{1mm} \left(\sigma m - \frac{1}{K} \underset{i = 1}{\stackrel{K}{\sum}} \hspace{1mm} \log \int_{\mathbb{R}}{\exp((\sigma + a_i)x - \psi(x))dx} \right).$$
We then have , for any $L > 0$ and any compact subset E of $\mathbb{R}$,
$$\underset{K \rightarrow \infty}{\lim} \hspace{1mm} \underset{a_1 .. a_K \in [-L, L]}{\sup}  \hspace{3mm} ||\psi_K - \varphi_K||_{\infty, E}  = 0.$$
In particular, if $a_i = \lambda(i/K)$ for some smooth function $\lambda$, then $\psi_K$ converges to 
$$\varphi_{\lambda}(m) := \underset{\sigma \in \mathbb{R}}{\sup} \hspace{1mm} \left(\sigma m - \int_0^1{\log \int_{\R}{\exp((\sigma + \lambda(\theta))x - \psi(x))dx}d\theta} \right).$$
uniformly on compact sets.
\end{thm}

A proof of this result can be found in [FM] or [K, Appendix A]. Roughly speaking, it says that local averages of a large number $K$ of spins behave like random variables satisfy a large deviation principle of speed $K$ and rate function $\varphi$.

The following proposition is a consequence of [GPV, Theorem 4.1]. 

\begin{prop} \label{prop_averaging}
Let $f_N$ be a sequence of probability densities with respect to $\mu_N$ which weakly converges to a deterministic profile $\rho$. Assume that
$$\frac{1}{N}\int{\frac{\langle A_0\nabla f, \nabla f \rangle}{f}d\mu_N} \leq C.$$
Then, for any smooth function $J : \T \rightarrow \R$ and bounded continuous function $F : \R^{2k+1} \rightarrow \R$, we have
$$\frac{1}{N}\int{\sum J(i/N)F(x_{i-k},..,x_{i+k})f_N(x)\mu_N(dx)} \longrightarrow \int_{\T}{J(\theta)\tilde{F}(\rho(\theta))d\theta}$$
where 
\begin{equation}
\tilde{F}(y) := \int{F(x_1,..,x_{2k+1})\mu^{\lambda, \otimes 2k+1}(dx)},
\end{equation}
with $\mu^{\lambda}(dx) = \frac{1}{Z}\exp(\lambda x - \psi(x))dx$ and $\lambda = \varphi'(y)$.
\end{prop}

Note that, in [GPV], it was also required that there exists a superlinear function $\omega$ such that $\int{\sum \omega(x_i) f(x)\mu_N(dx)} \leq CN$. However, under our assumptions on $\psi$, the bound on $\int{\frac{\langle A_0\nabla f, \nabla f \rangle}{f}d\mu_N}$ implies that $\int{\sum |x_i|^2 f(x)\mu_N(dx)} \leq CN$, as we have seen in Lemma \ref{lem_borne_l2}.

\begin{prop} \label{technical_lgs}
Let $\rho$ be a smooth function on the torus, and define the probability density with respect to $\mu$
$$G_N(x) = \frac{1}{Z}\exp \left( \underset{i = 1}{\stackrel{N}{\sum}} \hspace{1mm} \varphi'(\rho(i/N))x_i \right).$$
Then 

(i) The measures $G_N\mu_N$ weakly converge to the deterministic profile $\rho$;

(ii) They satisfy a logarithmic Sobolev inequality, with a constant that only depends on $\rho$ and $\psi$, but which is uniform in $N$;

(iii) For any sequence of probability measures $\mu_N$ on $\R^N$, if 
$$\frac{1}{N}\Ent_{G_N\mu_N}(\nu_N) \longrightarrow 0,$$
then the sequence weakly converges to the deterministic profile $\rho$. Moreover, we then have
$$\frac{1}{N}\Ent_{\mu_N}(\nu_N) \longrightarrow \int{\varphi(\rho)d\theta} - \varphi\left(\int{\rho d\theta}\right).$$
\end{prop}

\begin{proof}
(i) is a classic large deviation result. See for example [Y]. (ii) was proven in [FM]. (iii) is a consequence of these two results, and we can prove it as follows.

Since the measures $G_N\mu_N$ satisfy a logarithmic Sobolev inequality, they also satisfy a transport entropy inequality, that is
$$W_{2,A_0^{-1}}(\nu_N, G_N\mu_N)^2 \leq C\Ent_{G_N\mu_N}(\nu_N).$$
The fact that we can use the inner product given by $A_0$ rather than the usual inner product follows from the discrete Poincar\'e inequality. Therefore, we have
$$\frac{1}{N}W_{2,A_0^{-1}}(\nu_N, G_N\mu_N)^2 \longrightarrow 0.$$
The result then follows from the fact that $(G_N\mu_N)$ weakly converges to $\rho$, and that $\frac{1}{N}\langle A_0 x, x \rangle \leq C||\bar{x}||_{H^{-1}}^2$.

The second part is a consequence of the identity
$$\frac{1}{N}\Ent_{\mu_N}(\nu_N) = \frac{1}{N}\Ent_{G_N\mu_N}(\nu_N) + \frac{1}{N}\int{\log G_N d\nu_N}$$
and the convergence
\begin{align}
\frac{1}{N}\int{\log G_N d\nu_N} &= \frac{1}{N}\int{\underset{i = 1}{\stackrel{N}{\sum}} \hspace{1mm} \varphi'(\rho_i)x_i \nu_N(dx)} \notag  \\
&- \frac{1}{N}\log\int{\exp\left(\underset{i = 1}{\stackrel{N}{\sum}} \hspace{1mm} \varphi'(\rho(i/N))x_i - \psi(x_i)\right)dx} + \frac{1}{N}\log\int{\exp\left(\underset{i = 1}{\stackrel{N}{\sum}} \hspace{1mm} - \psi(x_i)\right)dx} \notag \\
&\longrightarrow \int{\varphi(\rho(\theta))d\theta} - \varphi \left(\int{\rho(\theta)d\theta}\right).
\end{align}
This last convergence follows from the convergence of $(\nu_N)$ to the deterministic profile $\rho$ and Theorem \ref{lct}. A complete proof is given in Lemma 7.1 of [K].
\end{proof}

\begin{prop} \label{conv_renforcees}
Let $f_N$ be a sequence of probability densities with respect to $\mu_N$ which weakly converges to a deterministic profile $\rho$. Assume that
$$\frac{1}{N}\int{\frac{\langle A_0\nabla f, \nabla f \rangle}{f}d\mu_N} \leq C$$
and
$$\frac{1}{N}\int{\sum \omega(x_i)f_N(x)\mu_N(dx)} \leq C.$$
Then, for any sequence $(J^N)$ of step functions on the torus that are constant on the intervals $[(i-1)/N, i/N)$ and which converges in $H^1$ to a function $J$, we have
\begin{equation} \label{conv_ren1}
\frac{1}{N}\int{\sum J^N(i/N)x_i f_N\mu_N(dx)} \longrightarrow \int_{\T}{J(\theta)\rho(\theta)d\theta}
\end{equation}
and 
\begin{equation} \label{conv_ren2}
\frac{1}{N}\int{\sum J^N(i/N)\psi'(x_i) f_N\mu_N(dx)} \longrightarrow \int_{\T}{J(\theta)\varphi'(\rho(\theta))d\theta}.
\end{equation}
\end{prop}

\begin{proof}
For the first part, notice that 
$$\left|\frac{1}{N}\int{\sum J^N(i/N)x_i f_N\mu_N(dx)} - \int{\int_{\T}{J(\theta)\bar{x}(\theta)d\theta}f_N(x)\mu_N(dx)}\right|$$
$$\leq ||J^N - J||_{H^1}\left(\int{||\bar{x}||_{H^{-1}}^2f_N(x)\mu_N(dx)}\right)^{1/2} \longrightarrow 0$$
and 
\begin{align}
&\left|\int{\int_{\T}{J(\theta)\bar{x}(\theta)d\theta}f_N(x)\mu_N(dx)}- \int_{\T}{J(\theta)\rho(\theta)d\theta}\right| \notag \\
&\hspace{1cm}\leq ||J||_{H^1}\left(\int{||\bar{x} - \rho||_{H^{-1}} f_N(x)\mu_N(dx)}\right) \notag \\
&\hspace{1cm} ||J||_{H^1}\left(\int{||\bar{x} - \rho||_{L^2} f_N(x)\mu_N(dx)}\right) \notag 
\end{align}
so we just have to show that 
$$\int{||\bar{x} - \rho||_{L^2} f_N(x)\mu_N(dx)} \longrightarrow 0.$$
This quantity is the Wasserstein distance $W_1$ between $f_N\mu_N$ and $\delta_{\rho}$ for the $L^2$ distance. Since we already have weak convergence, to show that there is convergence for $W_1$, according to [Vi1, Theorem 7.12], we just have to prove the following tightness estimate
$$\underset{R \rightarrow \infty}{\lim} \hspace{1mm} \underset{N \rightarrow \infty}{\limsup} \hspace{1mm} \int_{\sum |x_i|^2 \geq NR^2}{\sqrt{\frac{1}{N}\sum |x_i|^2} f_N(x)\mu_N(dx)} = 0.$$
This estimate automatically follows from the bound
$$\underset{N}{\sup} \frac{1}{N}\int{\sum |x_i|^2f_N(x)\mu_N(dx)} < +\infty$$
that was given by Lemma \ref{lem_borne_l2}.

For the second part, we give a very brief sketch of the method of proof that was used in [GPV]. Let $\psi_{\ell}$ be a cutoff of $\psi'$ at level $\ell > 0$, that is 
$$\psi_{\ell}(x) = \psi'(x) \text{ if } |\psi'(x)| \leq \ell, \psi_{\ell}(x) = \pm \ell \text{ if not.} $$
From the bound of Lemma \ref{lem_borne_lp}, we can deduce 
\begin{equation}
\frac{1}{N}\int{\sum J^N(i/N)\psi'_{\ell}(x_i) f_N\mu_N(dx)} \underset{\ell \longrightarrow \infty}{\lim} \hspace{1mm} \frac{1}{N}\int{\sum J^N(i/N)\psi'(x_i) f_N\mu_N(dx)}
\end{equation}
uniformly in $N$, since $\psi$ goes to infinity faster than $|\psi'|$. Moreover, from Proposition \ref{prop_averaging}, we obtain
\begin{equation}
\frac{1}{N}\int{\sum J^N(i/N)\psi'_{\ell}(x_i) f_N\mu_N(dx)} \underset{N \longrightarrow \infty}{\lim} \hspace{1mm}  \int_{\T}{J(\theta)\tilde{\psi}_{\ell}(\rho(\theta))d\theta},
\end{equation}
so all we need to do is show that $\tilde{\psi}_{\ell}$ converges to $\varphi'$, which was done in [GPV, Lemma 6.4].
\end{proof}

\begin{rmq}
Similarly, under the assumption that $\int_0^T{\int{\frac{\langle A_0 \nabla f, \nabla f \rangle}{f}f\mu_N}dt} \leq CN$ uniformly in $N$, then (\ref{conv_ren1}) and (\ref{conv_ren2}) hold in a time-integrated sense.
\end{rmq}

We also give a priori estimates on weak limits of sequences of probability measures, obtained as direct consequences of [GPV], Lemmas 6.3 and 6.6:

\begin{lem} \label{limites_dans_h1}
Under our assumptions on $\psi$, for any sequence of probability $f_N$ with respect to $\mu_N$ that weakly converges to a deterministic trajectory $\rho$, such that 
$$\underset{N}{\sup} \hspace{1mm} \frac{1}{N}\int{\frac{\langle A_0 \nabla f, \nabla f \rangle}{f} d\mu_N} \leq C$$
we have
$$\int_{\T}{\varphi(\rho(\theta))d\theta} \leq C$$
and
$$\int_{\T}{(\partial_{\theta}\varphi(\rho)(\theta))^2d\theta} \leq C.$$
\end{lem}

Finally, we shall need the following lower bound on the slope of absolutely continuous curves.

\begin{lem} \label{lem_h-1}
Let $(\nu_t)$ be an absolutely continuous curve of probability measures on $\R^n$, which is equipped with a Riemannian tensor $(A^{-1}(x))$ satisfying the assumptions (\ref{assump-tensor1}) and (\ref{assump-tensor2}). Then we have, for any smooth function $V : [0,T] \times \R^n \longrightarrow \R$, 
\begin{align}
\int{|\dot{\nu}_t|^2dt} &\geq 2\int{V(T,x)\nu_T(dx)} - 2\int{V(0,x)\nu_0(dx)} \notag \\
& -2\int_0^T{\int{\frac{\partial V}{\partial t}(t,x)\nu_t(dx)}dt} - \int_0^T{\int{\langle A(x)\nabla V, \nabla V \rangle \nu_t(dx)}dt}
\end{align}
\end{lem}

\begin{proof}
From [L, Theorem 2.4], we know that there exists a vector field $v_t$ such that 
\begin{equation} \label{celem}
\dot{\nu}_t + \text{div}(v_t\nu_t) = 0
\end{equation}
and
\begin{equation}
\int_0^T{|\dot{\nu}_t|^2dt} = \int_0^T{\int{\langle A^{-1}(x)v_t(x), v_t(x) \rangle \nu_t(dx)}dt}.
\end{equation}
Since we have 
$$\langle A^{-1}(x)v_t(x), v_t(x) \rangle \geq 2\langle v_t(x), \nabla V(t,x) \rangle - \langle A(x) \nabla V(t,x), \nabla V(t,x) \rangle$$
for any $t$ and $x$, we get
\begin{align}
\int_0^T{|\dot{\nu}_t|^2dt} &\geq 2\int_0^T{\int{\langle v_t(x), \nabla V(t,x) \rangle \nu_t(dx)}dt} \notag \\
 \hspace{1cm} - \int_0^T{\int{\langle A(x) \nabla V(t,x), \nabla V(t,x) \rangle\nu_t(dx)}dt} \notag 
\end{align}
Using (\ref{celem}) to do an integration by parts on the first term, the result immediately follows.
\end{proof}

\subsection{Large deviations for the GL model in a random environment}

In this section, we shall be interested in the large deviations for a version of the process (\ref{gl_eq}) in a random environment, where the operator $A$ is replaced by a realization of the symmetric random matrix

\begin{equation}
A_{i,j}(\omega) := N^2a_{i+1}(\omega)(\delta_{i, j-1} - \delta_{i, j}) - N^2 a_i(\omega)(\delta_{i, j} - \delta_{i, j+1})
\end{equation}
where the $a_i$ are iid random variable defined on a probability space $\Omega$, and we assume there exists a constant $c > 0$ such that we almost surely have

\begin{equation}
c \geq a_i \geq 1/c.
\end{equation}
This assumption corresponds to an ellipticity assumption on the operator $A$ that is uniform in the realization of the random field. Therefore, for any $x$ and any realization of the random field, we have
\begin{equation}
\frac{1}{c}\langle A_0 x, x \rangle \leq \langle Ax, x \rangle \leq c \langle A_0 x, x \rangle.
\end{equation}

Under these assumptions, the quantity
\begin{equation}
\bar{a}:= \mathbb{E}\left[\frac{1}{a_1}\right]
\end{equation}
is well defined and finite.

The associated SDE is 

\begin{equation} \label{sde_random}
dX^i_t = N^2\left(a_i(\psi(X_{i+1}) - \psi(X_i)) - a_{i-1} (\psi(X_i) - \psi(X_{i-1}))\right) +\sqrt{2a_i}dB^i_t - \sqrt{2a_{i-1}}dB^{i-1}_t.
\end{equation}

Given a realization of the random environment, we denote by $L_{a,N}$ the generator of this diffusion.

The following hydrodynamic limit result for the random environment model has been proven in [Fr]:

\begin{thm} Assume that the sequence of initial data $f_{0,N}\mu_N$ weakly converges to a deterministic profile $\rho_0 \in H^1(\T)$. Then, for any time $t > 0$ and any smooth function $J : \T \longrightarrow \R$, the random variable $\frac{1}{N}\sum J(i/N)X_t^i$ converges in probability to $\int_{\T}{J(\theta)\rho(t,\theta)d\theta}$, where $\rho(t,\theta)$ is the unique solution to the PDE
$$\frac{\partial \rho}{\partial t} = \bar{a}\Delta \varphi'(\rho)$$
with initial condition $\rho_0$. This convergence holds for almost every realization of the random field.
\end{thm}

We are interested in the following quenched large deviation principle for this model, using the gradient flow approach we developed in section 1. 

\begin{thm} \label{thm_ldp_rand} Assume that the sequence of initial data is of the form (\ref{init_GL}) for some smooth initial profile $\rho_0$. Then, for almost every realization of the random field, the sequence of random functions satisfies a LDP in $L^{\infty}(H^{-1})$ with speed $N$ and rate function
\begin{align}
I(\rho) &:= \int{\varphi(\rho(0,\theta)) - \varphi(m_0(\theta)) - \varphi'(m_0(\theta))(\rho(\theta) - m_0(\theta))d\theta} \notag \\
& \hspace{1cm} + \frac{1}{4\bar{a}}\int_0^T{\left|\left| \frac{\partial \rho}{\partial t} - \bar{a}\frac{\partial^2}{\partial \theta^2}\varphi'(\rho)\right|\right|_{H^{-1}}^2dt}. \notag
\end{align}
\end{thm}

This generalizes the large deviations principle of [DV] to the case of random environment.

In terms of gradient flows, this result follows from two facts : 

\begin{itemize}
\item The relative entropy with respect to the invariant measure, divided by $N$, $\Gamma$-converges to $\rho \longrightarrow \int{\varphi(\rho)} - \varphi\left(\int{\rho} \right)$. This corresponds to a large deviations principle for the sequence of invariant measures $\mu_N$;

\item The sequence of metrics given by $A^{-1}(w)$ almost surely converge to the $H^{-1}$ norm, divided by a factor $\bar{a}$.
\end{itemize}

As a technical tool, we shall need the following convergence result, which will be used to formalize the convergence of the discrete norms.

\begin{lem} \label{lem_random}
Let $(a_q)_{q \in \Q}$ be a sequence of positive, bounded, iid random variables, and let $a_i^N := a_{i/N}$. With probability one, for any sequence $(h^N)$ of step functions  on $\T$ that converges to a function $h$ in $L^1$, such that $h^N$ is constant on $\left(\frac{i-1}{N}, \frac{i}{N}\right]$, and denoting by $h^N_i$ the value of $h^N$ on such an interval, we have
$$\lim \frac{1}{N}\underset{i = 1}{\stackrel{N}{\sum}} \hspace{1mm} a_i^Nh_i^N = \mathbb{E}(a)\int{h(\theta)d\theta}.$$
\end{lem}

\begin{proof} Let $M$ be an integer. The strong law of large numbers implies that, with probability 1, for any step function $h$ that is constant on the intervals $\left(\frac{i-1}{M}, \frac{i}{M}\right]$, we have 
$$\lim \frac{1}{N}\underset{i = 1}{\stackrel{N}{\sum}} \hspace{1mm} a_i^Nh(i/N) = \mathbb{E}(a)\int{h(\theta)d\theta}.$$
This then remains true simultaneously for every integer $M$, still with probability 1. An approximation argument in $L^1(\T)$ then yields the desired result.
\end{proof}

\begin{proof} [Proof of Theorem \ref{thm_ldp_rand}]

Given a realization of our random environment, the functional $J_N$ is given by
\begin{equation}
J_N(\nu_t) = \Ent_{\mu}(\nu_T) - \Ent_{\mu}(\nu_0) + \frac{1}{2}\int_0^T{||\partial_t \nu_t||_{H^{-1}(A^{-1})}dt} + \frac{1}{2}\int_0^T{\int{\frac{\langle A\nabla g_t, \nabla g_t \rangle}{g_t} d\mu}dt}
\end{equation}
where $g_t$ is the density of $\nu_t$ with respect to $\mu$. 

\begin{lem} \label{gamma_con_ent}
The functional $\frac{1}{N}\Ent_{\mu_N}$ $\Gamma$-converges at every Dirac mass to $\rho \rightarrow \int{\varphi(\rho)d\theta} - \varphi \left(\int{\rho d\theta} \right)$.
\end{lem}

\begin{proof}
Let $\nu_N$ be a sequence that weakly converges to a deterministic profile $\rho : \T \rightarrow \R$, and let 
\begin{equation}
G_N(x) := \frac{1}{Z_N}\exp\left(\underset{i = 1}{\stackrel{N}{\sum}} \hspace{1mm} \varphi'(\rho_{i,N})x_i \right)
\end{equation}
be a local Gibbs profile with respect to $\mu$, where $\rho_{i,N} = \int_{(i-1)/N}^{i/N}{\rho(\theta)d\theta}$ and $Z_N$ is the normalization constant such that $G\mu$ is a probability measure.

We then have the decomposition

\begin{equation} \label{dec_ent}
\Ent_{G_N\mu}(\nu_N) = \Ent_{\mu}(\nu_N) - \int{\log G_N d\nu_N}.
\end{equation}

By definition of $G_N$, since $\nu_N$ weakly converges to the deterministic profile $\rho$, as we have seen in the proof of Proposition \ref{technical_lgs}, $\frac{1}{N}\int{\log G_N d\nu_N}$ converges to $\int{\varphi(\rho)d\theta} - \varphi \left(\int{\rho d\theta} \right)$. Since the relative entropy is nonnegative, we can deduce from (\ref{dec_ent}) the inequality
$$\liminf \hspace{1mm} \frac{1}{N} \Ent_{\mu}(\nu_N) \geq \int{\varphi(\rho)d\theta} - \varphi \left(\int{\rho d\theta} \right).$$

Moreover, the measures $G_N \mu$ weakly converge to the deterministic profile $\rho$, so that they provide the recovery sequence for this Gamma-convergence result.
\end{proof}

\begin{lem} \label{conv_ent_init}
Let $(\nu_N)$ be a sequence of probability measures that converges to a profile $\rho$, such that, for any $N$, $\nu_N$ is absolutely continuous mith respect to $\mu_N$.

Then
$$\frac{1}{N}\Ent_{f_{0,N}\mu_N}(\nu_N) - \frac{1}{N}\Ent_{\mu_N}(\nu_N) \longrightarrow \int_{\T}{\varphi'(m_0)(m_0 - \rho) -\varphi(m_0)d\theta}.$$
\end{lem}

\begin{proof}
Since $\nu_N$ is absolutely continuous with respect to $\mu_N$ (and therefore to $f_{0,N}\mu_N$), we have
\begin{align}
\frac{1}{N}&\Ent_{f_{0,N}\mu_N}(\nu_N) - \frac{1}{N}\Ent_{\mu_N}(\nu_N) \notag \\
&= -\frac{1}{N}\int{\log f_{0,N} d\nu_N} \notag \\
&= -\frac{1}{N}\int{\sum \varphi'(m_0(i/N))x_i \nu_N(dx)} + \frac{1}{N} \log \int{\exp\left(\sum \varphi'(m_0(i/N))x_i \right) \mu_N(dx)} \notag \\
&= -\frac{1}{N}\int{\sum \varphi'(m_0(i/N))x_i \nu_N(dx)} + \frac{1}{N} \log \int{\exp\left(\sum \varphi'(m_0(i/N))x_i - \psi(x_i) \right) dx} \notag \\
&- \hspace{5mm} \frac{1}{N} \log \int{\exp\left(\sum  - \psi(x_i) \right) dx} \notag \\
&\longrightarrow -\int{\varphi'(m_0(\theta))\rho(\theta)d\theta} + \int{\varphi'(m_0(\theta))m_0(\theta)d\theta} \notag \\
&\hspace{1cm} - \int{\varphi(m_0(\theta)d\theta} \notag
\end{align}
since $(\nu_N)$ has asymptotic profile $\rho$, and applying Theorem \ref{lct}.
\end{proof}

We will now investigate the behavior of the slope:

\begin{lem}[Lower bound for the time-derivative]
For any time $t$ and subsequence such that $$\underset{N}{\sup} \hspace{1mm} \frac{1}{N} J_N(f_N) < +\infty,$$
we have
$$\liminf \frac{1}{N}\int_0^T{|\dot{\nu}_{N,t}|^2_{\nu_{N,t}}dt} \geq \frac{1}{\bar{a}}\int_0^T{||\partial \rho/ \partial t ||_{H^{-1}}^2dt}.$$

\end{lem}

\begin{proof}

Let $J : [0, T] \times \T \longrightarrow \R$ be a smooth function. Applying Lemma \ref{lem_h-1} with $J_N(t,x) := \frac{1}{N}\underset{i}{\sum} J(t, i/N)x_i + \underset{j = 1}{\stackrel{i-1}{\sum}} \frac{\partial J}{\partial \theta}(t, j/N)b_j x_i $ and $b_i = \frac{\bar{a}}{a_i} - 1$, we have
\begin{align}
\frac{1}{N}&\int_0^T{|\dot{\nu}_{N,t}|^2_{\nu_{N,t}}dt} \geq \frac{2}{N}\int{\sum J(T,i/N)x_i\nu_{N,T}(dx)} + \frac{2}{N^2} \int{\underset{i}{\sum} \underset{j = 1}{\stackrel{i-1}{\sum}} \frac{\partial J}{\partial \theta}(T, j/N)b_j x_i \nu_{N,T}(dx)} \notag \\
& \hspace{5mm} - \frac{2}{N}\int{\sum J(0,i/N)x_i\nu_{N,T}(dx)} - \frac{2}{N^2} \int{\underset{i}{\sum} \underset{j = 1}{\stackrel{i-1}{\sum}} \frac{\partial J}{\partial \theta}(0, j/N)b_j x_i \nu_{N,0}(dx)}\notag \\
& \hspace{5mm} - \frac{2}{N}\int_0^T{\int{\sum \frac{\partial J}{\partial t}(t, i/N) x_i \nu_{N,t}(dx)}dt} - \frac{2}{N^2} \int_0^T{\int{\underset{i}{\sum} \underset{j = 1}{\stackrel{i-1}{\sum}} \frac{\partial^2 J}{\partial t \partial \theta}(t, j/N)b_j x_i \nu_{N,t}(dx)}dt}\notag \\
& \hspace{5mm} - \frac{1}{N}\int_0^T{\sum a_i\left(NJ(t, \frac{i+1}{N}) - NJ(t, \frac{i}{N}) + b_i \frac{\partial J}{\partial \theta}(t, i/N)\right)^2dt}. \notag
\end{align}
Taking the limit $N \longrightarrow +\infty$ and using the second-moment bounds of Lemma \ref{lem_borne_l2}, we get 
\begin{align}
\frac{2}{N}\int{\sum J(T,i/N)x_i\nu_{N,T}(dx)}& + \frac{2}{N^2} \int{\underset{i}{\sum} \underset{j = 1}{\stackrel{i-1}{\sum}} \frac{\partial J}{\partial \theta}(T, j/N)b_j x_i \nu_{N,T}(dx)} \notag \\
&= \frac{2}{N}\int{\sum J(T,i/N)x_i\nu_{N,T}(dx)} + O\left(\frac{1}{N}\right) \notag \\
& \longrightarrow \int_{\T}{J(T,\theta)\rho(T,\theta)d\theta}. \notag
\end{align}
In the same way,
$$\frac{2}{N}\int{\sum J(0,i/N)x_i\nu_{N,T}(dx)} + \frac{2}{N^2} \int{\underset{i}{\sum} \underset{j = 1}{\stackrel{i-1}{\sum}} \frac{\partial J}{\partial \theta}(0, j/N)b_j x_i \nu_{N,0}(dx)}$$
$$\longrightarrow \int_{\T}{J(0,\theta)\rho(0,\theta)d\theta}$$
and 
$$\frac{2}{N}\int_0^T{\int{\sum \frac{\partial J}{\partial t}(t, i/N) x_i \nu_{N,t}(dx)}dt} + \frac{2}{N^2} \int_0^T{\int{\underset{i}{\sum} \underset{j = 1}{\stackrel{i-1}{\sum}} \frac{\partial^2 J}{\partial t \partial \theta}(t, j/N)b_j x_i \nu_{N,t}(dx)}dt}$$
$$\longrightarrow \int_0^T{\int_{\T}{\frac{\partial J}{\partial t}(t,\theta)\rho(t,\theta)d\theta}dt}.$$
Finally, using Lemma \ref{lem_random}, we get
$$\frac{1}{N}\int_0^T{\sum a_i\left(NJ(t, \frac{i+1}{N}) - NJ(t, \frac{i}{N}) + b_i \frac{\partial J}{\partial \theta}(t, i/N)\right)^2dt} $$
$$\longrightarrow \int_0^T{\int_{\T}{\bar{a}\left(\frac{\partial J}{\partial \theta}(t,\theta)\right)^2d\theta}dt}.$$
Combining these lower bounds, we get
$$\liminf \frac{1}{N}\int_0^T{|\dot{\nu}_{N,t}|^2_{\nu_{N,t}}dt} \geq 2\int{J(T,\theta)\rho(T,\theta)d\theta} - 2\int_{\T}{J(0,\theta)\rho(0,\theta)d\theta}$$
$$-2\int_0^T{\int_{\T}{\frac{\partial J}{\partial t}(t,\theta)\rho(t,\theta)d\theta}dt} - \int_0^T{\int_{\T}{\bar{a}\left(\frac{\partial J}{\partial \theta}(t,\theta)\right)^2d\theta}dt}.$$
Taking the supremum over all smooth functions $J$ yields the result.
\end{proof}

\begin{lem}[Lower bound for the upper gradient]
For any time $t$ and subsequence such that $$\underset{N}{\sup} \hspace{1mm} \frac{1}{N} \int{\frac{\langle A_0(\nabla \nu_t + \nu_t\nabla H), (\nabla \nu_t + \nu_t\nabla H) \rangle}{\nu_t}} < +\infty,$$
we almost surely have
$$\underset{N}{\liminf} \hspace{1mm} \frac{1}{N} \int{\frac{\langle A(\nabla \nu_t + \nu_t\nabla H), (\nabla \nu_t + \nu_t\nabla H) \rangle}{\nu_t}} \geq \bar{a} \int_{\T}{(\partial_{\theta}\varphi'(\rho(t, \theta)))^2d\theta}.$$
\end{lem}

\begin{proof}
From Lemma \ref{limites_dans_h1}, we know that, under these assumptions, $\varphi'(\rho)$ lies in $H^1(\T)$. 

Let $J : \T \rightarrow \R$ be a smooth function, and define
\begin{equation}
G_{N}(x) := \exp\left( \underset{i = 1}{\stackrel{N}{\sum}}\left(\underset{j = 1}{\stackrel{i}{\sum}} \hspace{1mm} \frac{\bar{a}}{a_j}J(i/N)\right) x_i \right).
\end{equation}

Since the upper gradient takes value $+\infty$ when $\nu_N$ is not absolutely continuous with respect to $\mu_N$, and we are looking for a lower bound, we can assume without loss of generality that $\nu_N$ is absolutely continuous with respect to $\mu_N$, and therefore to $G_{N}\mu_N$. Let $g_{N,t}$ be the density of $\nu_{N,t}$ with respect to $\mu_N$.

We consider the quantity $\int{\frac{\langle A\nabla \left(\frac{g_N}{G_{N}}\right), \nabla \left(\frac{g_N}{G_{N}}\right)\rangle}{g_{N,t}/G_{N}} G_Nd\mu_N}$, which is nonnegative. We have

\begin{align}
\int&{\frac{\langle A\nabla \left(\frac{g_N}{G_{N}}\right), \nabla \left(\frac{g_N}{G_{N}}\right)\rangle}{g_{N,t}/G_{N}} G_{N}d\mu_N} \notag \\
&= \int{\frac{\langle A\nabla g_{N,t}, \nabla g_{N,t} \rangle}{g_{N,t}} d\mu_N} - 2 \int{\frac{\langle A \nabla g_{N,t}, \nabla G_{N} \rangle}{G_{N}} d\mu_N} \notag \\
& \hspace{1cm} + \int{\frac{\langle A\nabla G_{N}, \nabla G_{N} \rangle}{G_{N}^2} d\mu_N} \notag \\
&= \int{\frac{\langle A\nabla g_{N,t}, \nabla g_{N,t} \rangle}{g_{N,t}} d\mu_N} - 2 \int{\frac{\langle A\nabla H,  \nabla G_{N} \rangle}{G_{N}} g_{N,t}d\mu_N}  \notag \\
& \hspace{1cm} + \int{\frac{\langle A\nabla G_{N}, \nabla G_{N} \rangle}{G_{N}^2} d\mu_N}. \notag 
\end{align}

Therefore, we have

\begin{equation}
\int{\frac{\langle A\nabla g_{N,t}, \nabla g_{N,t} \rangle}{g_{N,t}} d\mu_N} \geq 2 \int{\frac{\langle A\nabla H,  \nabla G_{N} \rangle}{G_{N}} g_{N,t}d\mu_N}  - \int{\frac{\langle A\nabla G_{N}, \nabla G_{N} \rangle}{G_{N}^2} d\mu_N}
\end{equation}
for any realization of the random field, any $N$ and any $t$.

Applying Lemma \ref{lem_random}, we have
\begin{align} 
\frac{1}{N}\int{\frac{\langle A\nabla G_{N,t}, \nabla G_{N,t} \rangle}{G_{N,t}^2} d\mu_N} &= \frac{1}{N}\underset{i = 1}{\stackrel{N}{\sum}} \hspace{1mm} \frac{\bar{a}^2}{a_i}J(i/N)^2\notag \\
&\longrightarrow \bar{a}\int_{\T}{J(\theta)^2d\theta}.
\end{align}

We also have
\begin{align}
\frac{1}{N}\int{\frac{\langle A\nabla H,  \nabla G_{N} \rangle}{G_{N}} g_{N,t}d\mu_N} &= \frac{\bar{a}}{N}\int{\sum (\psi'(x_{i+1}) - \psi'(x_i))J(i/N)g_N(x)\mu_N(dx)} \notag \\
&= \frac{\bar{a}}{N}\int{\sum \psi'(x_i)(J((i-1)/N) - J(i/N))g_N(x)\mu_N(dx)} \notag \\
&= \frac{\bar{a}}{N}\int{\sum \psi'(x_i)J'(i/N)g_N(x)\mu_N(dx)} \notag \\
& \hspace{5mm} + O\left(\frac{1}{N^2}\int{\sum |\psi'(x_i)|g_N(x)\mu_N(dx)}\right) \notag \\
&\longrightarrow \bar{a}\int_{\T}{\varphi'(\rho(\theta))J'(\theta)d\theta}
\end{align}
Combining these two lower bounds and taking the supremum over smooth functions $J$, we get the lower bound of our Lemma.
\end{proof}

From the previous Lemma and Fatou's Lemma, we can then deduce that, for a sequence that converges to a Dirac mass, and such that $J_N(f_N) \leq CN$, we have
$$\liminf \frac{1}{N}\int_0^T{\int{\frac{\langle A(\nabla \nu_t + \nu_t\nabla H), (\nabla \nu_t + \nu_t\nabla H) \rangle}{\nu_t}}dt} \geq \int_0^T{\bar{a} \int_{\T}{(\partial_{\theta}\varphi'(\rho(t, \theta)))^2d\theta}dt}$$
which was the last element we needed for the lower bound of the $\Gamma$-convergence.

We now turn to the recovery sequence. Given a profile $\rho(t,\theta)$ that is weakly continuous in time, and such that $\int_0^T{||\bar{a}\partial_{\theta}^2\varphi'(\rho) - \partial_t \rho ||_{H^{-1}}^2 dt}$ is finite, there exists a sequence of smooth profiles $\rho_k$ that converge to $\rho$, and such that $\int_0^T{||\bar{a}\partial_{\theta}^2\varphi'(\rho_k) - \partial_t \rho_k ||_{H^{-1}}^2 dt}$ converges to $\int_0^T{||\bar{a}\partial_{\theta}^2\varphi'(\rho) - \partial_t \rho ||_{H^{-1}}^2 dt}$. Therefore, in view of Remark \ref{rq1}, we only have to prove the existence of a recovery sequence for profiles $\rho$ that are smooth.

Given such a smooth profile $\rho$, there exists a continuous function $h : [0, T] \times \T \longrightarrow \R$ such that 
\begin{equation} \label{eq_recov_rand}
\frac{\partial \rho}{\partial t} = \bar{a}\frac{\partial}{\partial \theta}\left(h(t,\theta) + \frac{\partial}{\partial \theta} \varphi'(\rho) \right).
\end{equation}

We now consider a dynamic with law given by the time-dependent generator 
\begin{equation}
\tilde{L}_{a,N} = L_{a,N} + N\underset{i = 1}{\stackrel{N}{\sum}} \hspace{1mm} \frac{\bar{a}}{a_i}h(t,i/N)\left(\frac{\partial}{\partial x_{i+1}} - \frac{\partial}{\partial x_i}\right)
\end{equation}
and initial condition given by the local Gibbs state associated to $\rho(0,\cdot)$.

We need to prove two things : that the solutions to such dynamics converge to the deterministic profile $\rho$, and that $N^{-1}J_N(f_N)$ has the correct limit (where $f_N$ are the marginals of the law of the solution). The first part can be done in the same way as in [Q, Section 3], so we concentrate on the second part. We have
\begin{align}
\frac{1}{N}J_N(f_N) &= \frac{1}{2}\int_0^T{\left|\left|\frac{\partial f_N}{\partial t} - L_{a,N}f_N \right|\right|_{H^{-1}(A)}^2dt} \notag \\
&= \frac{1}{2}\int_0^T{\left|\left|\tilde{L}_{a,N}f_N - L_{a,N}f_N \right|\right|_{H^{-1}(A)}^2dt} \notag \\
&= \frac{1}{2}\int_0^T{\underset{i = 1}{\stackrel{N}{\sum}} \hspace{1mm} \frac{\bar{a}^2}{a_i}h(t,i/N)^2 dt} \notag \\
&\longrightarrow \frac{1}{2}\int_0^T{\int_{\T}{\bar{a}h(t,\theta)^2d\theta}dt}
\end{align}
and, using (\ref{eq_recov_rand}), it is easy to see that this is equal to $\frac{1}{2\bar{a}}\int_0^T{\left|\left| \frac{\partial \rho}{\partial t} - \bar{a}\frac{\partial^2}{\partial \theta^2}\varphi'(\rho)\right|\right|_{H^{-1}}^2dt}$, which was what we needed to prove.

We still have to prove exponential tightness for the laws of solutions to (\ref{sde_random}). It is given by the following two results: 

\begin{lem} \label{lem_tightness1} Let $\mathbb{P}_N$ be the law of a solution to the SDE (\ref{sde_random}) with initial condition $X_0$ having a distribution $f_0\mu_N$ that satisfies $\Ent_{\mu_N}(f_0) \leq CN$. Then
$$\underset{\ell \longrightarrow +\infty}{\lim} \hspace{1mm} \underset{N \longrightarrow +\infty}{\lim} \hspace{1mm} \frac{1}{N}\log \mathbb{P}_N\left(\underset{0 \leq t \leq T}{\sup} \hspace{1mm} \frac{1}{N} \sum |X^i_t| \geq \ell \right) = -\infty.$$

\end{lem}

\begin{lem}\label{lem_tightness2} 
Under the same assumptions as the previous lemma, for any $\epsilon > 0$ and any smooth function on the torus $J$, we have
$$\underset{\delta \longrightarrow 0}{\lim} \hspace{1mm} \underset{N \longrightarrow +\infty}{\lim} \hspace{1mm} \frac{1}{N}\log \mathbb{P}_N\left( \underset{0 \leq s \leq t \leq T, |s - t| \leq \delta}{\sup} \hspace{2mm} \left|\frac{1}{N}\sum J(i/N)(X^i_t - X^i_s)\right| \geq \epsilon \right) = -\infty.$$
\end{lem}

\begin{proof} [Proof of Lemma \ref{lem_tightness1}] 

This proof is exactly the same as in [DV], we give a brief sketch to show that the random field $(a_i)$ does not make any difference..

Let $\tilde{P}^{eq, N}$ be the law of a solution to the SDE starting from the equilibrium measure $\mu_N$. From [KV, Lemma 1.12], we know that, for any symmetric function $g$ on $\R^N$, we have
$$\tilde{P}^{eq, N}\left(\underset{0 \leq t \leq T}{\sup} g(X_t) \geq \ell \right) \leq \frac{3}{\ell} \sqrt{a + Tb}$$
with $a = \int{g^2d\mu_N}$ and $b = \int{\frac{A\nabla g, \nabla g}{g}d\mu_N}.$
When $g(x) = \exp \left( \sum |x_i| \right)$, there exists $C > 0$ such that $a \leq C^N$ and $b \leq N^2 C^N$. Using the Tchebychev inequality, we obtain 
$$\tilde{P}^{eq, N}\left(\underset{0 \leq t \leq T}{\sup} \frac{1}{N} \sum |X^i_t| \geq \ell \right) \leq \sqrt{C^N(1 + N^2)}e^{-N\ell} \leq C'e^{-C''N\ell}.$$
We can then use the basic entropy inequality
$$ \mathbb{P}_N(A) \leq \frac{\log(2) + H_N}{\log(1 + 1/\tilde{P}^{eq, N}(A))},$$
where $H_N$ is the relative entropy of the non equilibrium process with respect to the equilibrium process, which satisfies the bound $H_N \leq CN$, to get the result.
\end{proof}

In the same way, we refer the reader interested in the proof of Lemma \ref{lem_tightness2} to [DV, Lemma 2.8]. The proof is exactly the same, since the random variables $a_i$ are bounded.
\end{proof}

\subsection{Large deviations for the non-gradient Ginzburg-Landau model}

We consider the SDE given by
\begin{equation} \label{sde_nongrad}
dX^i_t = N^2(W_{i, i+1} - W_{i-1, i})dt +N\sqrt{2a(X^i_t, X^{i+1}_t)}dB^{i+1}_t - N\sqrt{2a(X^{i-1}_t, X^i_t)}dB^i_t
\end{equation}
where 
\begin{equation}
W_{i, i+1}= a(X^i_t, X^{i+1}_t)(\psi'(X^i_t) - \psi'(X^{i+1}_t)) -\frac{\partial a}{\partial x}(X^i_t, X^{i+1}_t) + \frac{\partial a}{\partial y}(X^i_t, X^{i+1}_t).
\end{equation}
The marginals of the law of a solution to this SDE solve the PDE
\begin{equation}
\frac{\partial f\mu_N}{\partial t} = \nabla \cdot(A(x)\nabla f\mu_N)
\end{equation}
where the matrix $A(x)$ is given by
$$A(x)_{i,j}:= a(x_{i-1}, x_i)(\delta_{i,j} - \delta_{i, j+1}) + a(x_i, x_{i+1})(\delta_{i,j} - \delta_{i, j-1}).$$

Once more, we will assume that the initial data $f_0$ is of the form (\ref{init_GL}).

The generator of this dynamic is given by
\begin{equation}
Lf = N^2\sum e^{H}\left(\frac{\partial}{\partial x_{i+1}} - \frac{\partial}{\partial x_{i+1}} \right)e^{-H}a(x_i, x_{i+1})\left(\frac{\partial}{\partial x_{i+1}} - \frac{\partial}{\partial x_{i+1}} \right)f
\end{equation}

It has been shown in [Va] that trajectories of such a dynamic concentrate around the solution to the PDE

\begin{equation}
\frac{\partial \rho}{\partial t} = \frac{\partial}{\partial \theta}\left(\hat{a}(\rho)\frac{\partial}{\partial \theta}\varphi'(\rho) \right)
\end{equation} 
where $\varphi$ is the same function as in the previous section, and $\hat{a}$ is a bounded continuous function, which we shall now define.

Let $F : \R^{2k+1} \longrightarrow \R$ be a bounded, smooth function of a finite number of variables. The function $\xi : x \in \R^{\infty} \longrightarrow \underset{i = -\infty}{\stackrel{+\infty}{\sum}} \hspace{1mm} F(x_{i-k},.., x_{i+k})$ is not well-defined, but its partial derivatives are. We can therefore define
\begin{equation}
a_F(y) := \int{a(x_0, x_1)\left(1- \frac{\partial\xi}{\partial x_1} + \frac{\partial\xi}{\partial x_0}\right)^2\mu^{\infty, y}(dx)}
\end{equation}
where $\mu^{\infty, y}$ is the product measure on $\R^{\infty}$ with every one-dimensional marginal having density $Z^{-1}\exp(\lambda x - \psi(x))$, with $\lambda$ the unique real number such that this measure has expectation $y$.

The function $\hat{a}$ is then given by
\begin{equation}
\hat{a}(y) := \underset{F}{\inf} \hspace{1mm} a_F(y)
\end{equation}
with the infimum running over the set of all smooth, bounded functions of a finite number of variables.

In the next proposition, we summarize a few properties of the function $\hat{a}$: 

\begin{prop} \label{prop_coeff_diff}
(i) $\hat{a}$ is a bounded, continuous function.

(ii) For any $\epsilon > 0$ and $C < +\infty$, there exists a smooth real-valued function $g(x_k,..,x_k,y)$ on $\R^{2k+2}$ with bounded first derivatives such that 
$$\underset{|y| \leq C}{\sup} \hspace{1mm} \left( a_{g(\cdot, y)}(y) - \hat{a}(y) \right) < \epsilon$$
and
$$\underset{y \in \R}{\sup} \hspace{1mm} \left( a_{g(\cdot, y)}(y) - \hat{a}(y) \right) \leq ||a||_{\infty}.$$
\end{prop}

Part (i) of this Proposition comes from [Va], and part (ii) from [Q]. 

Our aim is to prove the following large deviations result: 

\begin{thm} \label{thm_ld_nongradient}
Assume that the sequence of initial data is of the form (\ref{init_GL}) for some smooth initial profile $\rho_0$. The sequence of random functions satisfies a LDP in $L^{\infty}(H^{-1})$ with speed $N$ and rate function
\begin{align}
I(\rho) &:= \int{\varphi(\rho(0,\theta)) - \varphi(m_0(\theta)) - \varphi'(m_0(\theta))(\rho(\theta) - m_0(\theta))d\theta} \notag \\
& \hspace{1cm} + \frac{1}{4}\int_0^T{\left|\left| \frac{\partial \rho}{\partial t} - \frac{\partial}{\partial \theta}\left(\hat{a}(\rho)\frac{\partial}{\partial \theta}\varphi'(\rho)\right)\right|\right|_{H^{-1}(\hat{a}(\rho(t,\cdot)))}^2dt}. \notag
\end{align}

\end{thm}

In the rate function, the norm is defined as
\begin{equation}
||u||_{H^{-1}(\hat{a}(\rho(t,\cdot)))}^2 := \underset{v \in H^1(\T)}{\sup} \hspace{1mm} 2\int_{\T}{u(\theta)v(\theta)d\theta} - \int{\hat{a}(\rho(t,\theta))\left(\frac{\partial v}{\partial \theta}\right)^2d\theta}.
\end{equation}

This result was already proved in [Q], under the assumption that the single site potential $\psi$ is uniformly convex, and that its second derivative is bounded above. Our assumptions are a priori more general, since they allow for superquadratic potentials, but it seems likely that, using the logarithmic Sobolev inequality proved in [MO], the method of [Q] could be extended for such functions.

The following result is the key technical estimate to prove large deviations for nongradient models. It is has been proven in [Va].

Let $R_N$ be the law of the stationary solution to the SDE (\ref{sde_nongrad}), with the initial condition $X_0$ having law $\mu_N$. Let $\rho(t,\theta)$ be a deterministic profile, which we assume to be in $L^{\infty}(H^1)$. 

Denote by $\rho_{\theta}^c(t) := \int_{\theta - c}^{\theta + c}{\rho(t,s)ds}$. For a given smooth function $J : [0,T] \times \T \rightarrow \R$ and a function $g$ as in part (ii) of Proposition \ref{prop_coeff_diff}, we define
\begin{align}
V(t):= & \underset{i = 1}{\stackrel{N}{\sum}} \hspace{1mm} J(t, i/N)\left[W_{i,i+1} - \frac{1}{N^2}Lg\left(X^{i-k}_t,..,X^{i+k}_t, \rho_{i/N}^{\ell/N}(t)\right)\right] \notag \\
&+ \frac{1}{N}\underset{i = 1}{\stackrel{N}{\sum}} \hspace{1mm} J(t, i/N)\hat{a}(\rho_{i/N}^{\epsilon_1}(t))\left(\frac{\varphi'(\rho_{i/N + \epsilon_2}^{\epsilon_1}(t)) - \varphi'(\rho_{i/N - \epsilon_2}^{\epsilon_1}(t))}{2\epsilon_2}\right) \notag \\
&-\frac{\alpha}{N}\underset{i = 1}{\stackrel{N}{\sum}} \hspace{1mm} J(t, i/N)^2(a_{g(\cdot, \rho_{i/N}^{\epsilon_1}(t))}(\rho_{i/N}^{\epsilon_1}(t)) - \hat{a}(\rho_{i/N}^{\epsilon_1}(t))
\end{align}
where $\alpha$, $\epsilon_1$ and $\epsilon_2$ are positive numbers, and $\ell$ is a positive integer.

Under these notations, we have the following exponential estimate, which is due to Varadhan [Va] : 

\begin{thm} \label{grad_replac_var}
For any profile $\rho$, any $\alpha > 0$, any $J$ and $g$, we have
$$\underset{\epsilon_2 \rightarrow 0}{\lim} \hspace{1mm} \underset{\epsilon_1 \rightarrow 0}{\limsup} \hspace{1mm} \underset{\ell \rightarrow \infty}{\limsup} \hspace{1mm} \underset{N \rightarrow \infty}{\limsup} \hspace{1mm} \frac{1}{N}\log \mathbb{E}^{R_N}\left[\exp\left(\alpha N\int_0^T{V(t)dt}\right)\right] \leq 0.$$

\end{thm}

Using the result, we obtain what is known in the hydrodynamic literature as the gradient replacement estimate.

\begin{cor} \label{grad_replac}
Let $(\nu_{N,t})$ be a sequence of flows of time-marginals of a a sequence of processes  that weakly converge to a deterministic flow $\rho(t,\theta)$. Assume moreover that $J_N(\nu_{N,t}) \leq CN$ for some $C > 0$. 

Then for any smooth, bounded functions $J: [0,T]\times \T \rightarrow \R$ and $g : \R^{2k+2} \rightarrow \R$ , we have
$$\underset{\ell \rightarrow \infty}{\lim} \hspace{1mm} \underset{N}{\lim} \hspace{1mm} \int_0^T{\int{\sum J(t,i/N)(W_{i, i+1}(x) - \frac{1}{N^2}(Lg)(x_{i-k},..,x_{i+k}, \rho_{i/N}^{\ell/N}))\nu_{N,t}(dx)}dt} $$
$$=  \int_0^T{\int_{\T}{J(t,\theta)\hat{a}(\rho(t,\theta))\partial_{\theta}\varphi'(\rho(t,\theta))d\theta}dt}.$$
\end{cor}

\begin{proof}
First, let us show that we can build a diffusion process of law $P$ with time-marginals $(\nu_{N,t})$ such that $H(P_N, R_N) = O(N)$. Since the equilibrium process is a solution of (\ref{sde_nongrad}) with initial condition $\mu_N$, we know that there exists a process with marginals $(\nu_{N,t})$ such that 
\begin{align}
H(P_N, R_N) &= \frac{1}{2}\Ent_{\mu_N}(\nu_{0,N} + \frac{1}{2}\Ent_{\mu_N}(\nu_{N,T}) + \frac{1}{4}\int_0^T{\int{\frac{\langle A(\nabla \nu_t + \nu_t\nabla H), (\nabla \nu_t + \nu_t\nabla H) \rangle}{\nu_t}}dt} \notag \\
\hspace{1cm}  &+ \frac{1}{4}\int_0^T{|\dot{\nu}_t|^2dt} \notag \\
&= J_N(\nu_{N,t}) + \Ent_{\mu_N}(\nu_{N,0}) - \Ent_{f_{0,N}\mu_N}(\nu_{N,0}) \notag 
\end{align}
so we only have to get an upper bound on $\Ent_{\mu_N}(\nu_{N,0}) - \Ent_{f_{0,N}\mu_N}(\nu_{N,0})$. Moreover, the bound on $J_N(\nu_{N,t})$ implies that $\Ent_{f_{0,N}\mu_N}(\nu_{N,0}) \leq CN$.

Denoting $\nu_{N,0} = \rho_N\mu_N$, we have
\begin{align}
\Ent_{\mu_N}(\nu_{N,0}) &- \Ent_{f_{0,N}\mu_N}(\nu_{N,0}) = \int{\rho \log f_{0,N} d\mu_N} \notag \\
&= \int{\sum \varphi'(m_0(i/N))x_i \nu_{N,0}(dx)} - \log Z_N \notag \\
&\leq CN \times \sqrt{\frac{1}{N}\int{\sum |x_i|^2\nu_{N,0}(dx)} } + CN \notag \\
\leq CN
\end{align}
where the last bound follows from $\Ent_{f_{0,N}\mu_N}(\nu_{N,0}) \leq CN$, Lemma \ref{lem_borne_l2} and the fact that the measures $f_{0,N}\mu_N$ satisfy a LSI with uniform constant, and have uniformly bounded second moments.

Consequently, using the entropy inequality, we have
\begin{align}
\mathbb{E}^{P_N}\left[\int_0^T{V(t)dt}\right] \leq \frac{C}{\alpha} + \frac{1}{\alpha N}\log \mathbb{E}^{R_N}\left[\exp\left(\alpha N\int_0^T{V(t)dt}\right)\right]
\end{align}
and the result immediately follows from Theorem \ref{grad_replac_var} and using the fact that the inequality is valid for both $J$ and $-J$.
\end{proof}

We also recall the exponential tightness estimates that have been proven in [Va], and which we need to apply Corollary \ref{cor_ent}: 

\begin{lem}
Let $\mathbb{P}_N$ be the law of a solution to the SDE \ref{sde_nongrad} with initial condition $X_0$ having a distribution $f_0\mu_N$ that satisfies $\Ent_{\mu_N}(f_0) \leq CN$. Then
$$\underset{\ell \longrightarrow +\infty}{\lim} \hspace{1mm} \underset{N \longrightarrow +\infty}{\lim} \hspace{1mm} \frac{1}{N}\log \mathbb{P}_N\left(\underset{0 \leq t \leq T}{\sup} \hspace{1mm} \frac{1}{N} \sum |X^i_t| \geq \ell \right) = -\infty.$$
and for any $\epsilon > 0$ and any smooth function on the torus $J$, we have
$$\underset{\delta \longrightarrow 0}{\lim} \hspace{1mm} \underset{N \longrightarrow +\infty}{\lim} \hspace{1mm} \frac{1}{N}\log \mathbb{P}_N\left( \underset{0 \leq s \leq t \leq T, |s - t| \leq \delta}{\sup} \hspace{2mm} \left|\frac{1}{N}\sum J(i/N)(X^i_t - X^i_s)\right| \geq \epsilon \right) = -\infty.$$
\end{lem}

The proof of these estimates is exactly the same as for the gradient case studied in [DV]. Once more, the use of a variable function $a(x_i, x_{i+1})$ does not make a difference as long as it is bounded.

Lemmas \ref{gamma_con_ent} and \ref{conv_ent_init} remain valid, so that, to prove Theorem \ref{thm_ld_nongradient}, we only have to study the behavior of the slopes.

\begin{lem}
$$\underset{N}{\liminf} \hspace{1mm} \frac{1}{N} \int_0^T{|\dot{\nu}_t|^2dt} \geq \int_0^T{||\partial_t \rho||_{H^{-1}(\hat{a})}^2dt}$$
\end{lem}

\begin{proof}
Let $J : [0,T]\times \T \longrightarrow \R$ be a smooth function, and $F : \R^{2k+1} \longrightarrow \R$ be a smooth, bounded function. Applying Lemma \ref{lem_h-1} with $V(t,x) = \sum J(t,i/N)x_i + \frac{1}{N}J'(t,i/N)F(x_{i-k},.., x_{i+k})$, we get 
\begin{align}
\frac{1}{N} \int_0^T{|\dot{\nu}_t|^2dt} &\geq 2\frac{1}{N}\int{V(T,x)\nu_T(dx)} - 2\frac{1}{N}\int{V(0,x)\nu_0(dx)} \notag \\
&\hspace{5mm} - 2\frac{1}{N}\int_0^T{\int{\frac{\partial V}{\partial t}(x)\nu_t(dx)}dt} \notag \\
&\hspace{1cm} - \frac{1}{N}\int_0^T{\int{\langle A(x)\nabla V(t,x), \nabla V(t,x) \rangle \nu_t(dx)}dt}.
\end{align}

We have
\begin{align}
\frac{1}{N}\int{V(T,x)\nu_T(dx)} &= \frac{1}{N}\int{J(T,i/N)x_i \nu_T(dx)} + \frac{1}{N^2}\int{J'(T,i/N)F(x_{i-k},..,x_{i+k}) \nu_T(dx)} \notag \\
&= \frac{1}{N}\int{J(T,i/N)x_i \nu_T(dx)} + O\left( \frac{1}{N} \right) \notag \\
& \longrightarrow \int{J(T,\theta)\rho(T,\theta)d\theta}.
\end{align}

In the same way,
\begin{equation}
\frac{1}{N}\int{V(0,x)\nu_0(dx)} \longrightarrow \int{J(0,\theta)\rho(0,\theta)d\theta}
\end{equation}
and
\begin{equation}
\frac{1}{N}\int_0^T{\int{\frac{\partial V}{\partial t}(x)\nu_t(dx)}dt} \longrightarrow \int_0^T{\int_{\T}{\frac{\partial J}{\partial t}(t,\theta)\rho(t,\theta)d\theta}dt}.
\end{equation}
For the last term, we have
\begin{align}
&\frac{1}{N}\int_0^T{\int{\langle A(x)\nabla V(t,x), \nabla V(t,x) \rangle \nu_t(dx)}dt} \notag \\
&= \frac{1}{N}\int_0^T\int\underset{i}{\sum} N^2a(x_i, x_{i+1})\Bigg(J(t,(i+1)/N) - J(t,i/N)  \notag \\
& \hspace{5mm} + \left. \frac{1}{N}\underset{j = i-k}{\stackrel{i+k}{\sum}} \hspace{1mm} \frac{\partial}{\partial x_{i+1}}F(x_{j-k},..,x_{j+k}) - \frac{\partial}{\partial x_i}F(x_{j-k},..,x_{j+k}) \right)^2\nu_t(dx)dt \notag \\
&= \frac{1}{N}\int_0^T{\int{\underset{i}{\sum} a(x_i, x_{i+1})J'\left(t,\frac{k}{N}\right)^2\left(1 - \left(\frac{\partial}{\partial x_{i+1}} - \frac{\partial}{\partial x_i} \right) \underset{j = i-k}{\stackrel{i+k}{\sum}} F(x_{j-k},..,x_{j+k}) + O\left(\frac{k}{N}\right)\right)^2\nu_t(dx)}dt} \notag \\
&\longrightarrow \int_0^T{\int_{\T}{a_F(\rho(t,\theta))\left(\frac{\partial J}{\partial \theta} \right)^2d\theta}dt}.
\end{align}

We combine these convergence estimates, and then optimize in $F$ and $J$ to get the desired result.

\end{proof}

\begin{lem}For a flow of marginals $\nu_{N,t}$ that weakly converges to $\rho$, and such that $\frac{1}{N}J_N((\nu_{N,t})_t)$ is bounded, we have
$$\underset{N}{\liminf} \hspace{1mm} \frac{1}{N} \int_0^T{\int{\frac{\langle A(\nabla \nu_t + \nu_t\nabla H), (\nabla \nu_t + \nu_t\nabla H) \rangle}{\nu_t}}dt} \geq \int_0^T{\int_{\T}{\hat{a}(\rho(t,\theta))(\partial_{\theta}\varphi'(\rho))^2d\theta}dt}.$$
\end{lem}

\begin{proof}
Let $J(t,\theta)$ be a smooth function, $F : \R^{2k+1} \rightarrow \R$ a smooth function and let $\xi(x) = \sum F(x_{i-k},..,x_{i+k})$. We define$\vec{J}_N(t,x)$ the element of $\R^N$ given by $\vec{J}_N(t,x)_i := \underset{j = 1}{\stackrel{i-1}{\sum}} \hspace{1mm} J'(t,j/N) + \frac{\partial}{\partial x_i}\left[\frac{1}{N}\underset{j = 1}{\stackrel{N}{\sum}} \hspace{1mm} J'(t,j/N)F(x_{j-k},.., x_{j+k})\right]$. We have
\begin{align}
\frac{1}{N}& \int_0^T{\int{\frac{\langle A(\nabla \nu_t + \nu_t\nabla H), (\nabla \nu_t + \nu_t\nabla H) \rangle}{\nu_t}}dt} \notag \\
&\geq \frac{2}{N} \int_0^T{\int{\langle A(\nabla \nu_t + \nu_t\nabla H), \vec{J} \rangle}dt} - \frac{1}{N}\int_0^T{\int{\langle A\vec{J}, \vec{J} \rangle \nu_t(dx)}dt} \notag \\
&= 2\int_0^T{\int{\sum J'(t,i/N)W_{i, i+1}(x) \nu_t(dx)}dt} \notag \\
&\hspace{5mm} - 2\int_0^T{\int{\sum \frac{J'(t,i/N)}{N^2}(LF)(x_{i-k},..,x_{i+k})\nu_t(dx)}dt} \notag \\
&\hspace{1cm} - \int_0^T{\int{\sum a(x_i, x_{i+1})J'(t,i/N)^2(1 - \frac{\partial \xi}{\partial x_{i+1}} + \frac{\partial \xi}{\partial x_i})^2  \nu_t(dx)}dt} + o(1)
\end{align}
We then just have to use Corollary \ref{grad_replac} and optimize in $F$ to obtain
$$\liminf \frac{1}{N} \int_0^T{\int{\frac{\langle A(\nabla \nu_t + \nu_t\nabla H), (\nabla \nu_t + \nu_t\nabla H) \rangle}{\nu_t}}dt} $$
$$\geq 2\int_0^T{\int_{\T}{\varphi'(\rho(t,\theta))J'(t,\theta)d\theta}dt} - \int_0^T{\int_{\T}{\hat{a}(\rho)J(t,\theta)^2d\theta}dt}.$$
Taking the supremum over all smooth functions $J$ then yields our Lemma.

For the Gamma-convergence upper bound, the method we use is pretty much the same as the proof of the LDP lower bound in [Q], so we only give a rough sketch. We fix a smooth profile $\rho$, for which there exists a continuous function $h : [0,T] \times \T  \longrightarrow \R$ such that
\begin{equation} \label{edp_non_grad}
\frac{\partial \rho}{\partial t} = \frac{\partial}{\partial \theta}\left(\hat{a}(\rho)\left(\frac{\partial}{\partial \theta}\varphi'(\rho) + h \right)\right).
\end{equation}
We consider an evolution given by the generator
$$\tilde{L}_Nf := L_Nf + N\underset{i = 1}{\stackrel{N}{\sum}} \hspace{1mm} h(t,i/N)a(x_i,x_{i+1})\left(1 + \frac{\partial}{\partial x_{i+1}}\xi - \frac{\partial}{\partial x_i}\xi\right)\left( \frac{\partial}{\partial x_{i+1}}f - \frac{\partial}{\partial x_i}f \right).$$
and initial condition the local Gibbs state associated to $\rho(0,\cdot)$. It is shown in [Q, Section 3] that the solutions converge to the deterministic profile $\rho$.
We write $f_{N,F}$ the law of the solution. We have
\begin{align}
\frac{1}{N}J_N(f_N) &= \frac{1}{2N}\int_0^T{\int{a(x_i,x_{i+1})h(t,i/N)^2\left(1 + \frac{\partial}{\partial x_{i+1}}\xi - \frac{\partial}{\partial x_i}\xi\right)^2f_N(dx)}dt} \notag \\
&\longrightarrow \frac{1}{2}\int_0^T{\int_{\T}{a_F(\rho)\left(\frac{\partial h}{\partial \theta} \right)^2d\theta}dt}
\end{align}
If, instead of using a fixed function $F$, we use a sequence $(F_N)$ such that $a_{F_N}(\rho)$ uniformly converges to $\hat{a}(\rho)$ on compact sets (which is possible, see Proposition \ref{prop_coeff_diff}), we obtain the upper bound
$$\limsup \frac{1}{N}J_N(f_{N,F_N}) \leq \frac{1}{2}\int_0^T{\int_{\T}{\hat{a}(\rho)\left(\frac{\partial h}{\partial \theta} \right)^2d\theta}dt}$$
which, by representation (\ref{edp_non_grad}), is the one we needed to prove the upper bound in the Gamma convergence.
\end{proof}

\vspace{1cm}

\textbf{Acknowledgments:} I would like to thank Christian L\'eonard for having pointed out to me the use of Girsanov's theorem to understand relative entropy for diffusion processes. I would also like to thank Thierry Bodineau, Georg Menz, Felix Otto, S.R.S. Varadhan and C\'edric Villani for discussions about gradient flows and hydrodynamic limits.
\vspace{1cm}

\textbf{Bibliography}

\begin{itemize}

\item
\label{ADPZ1}[ADPZ1]
	S. Adams, N. Dirr, M. A. Peletier, and J. Zimmer. 
	From a large-deviations principle to the Wasserstein gradient flow: a new micro-macro passage. 
	Communications in Mathematical Physics, 307:791--815, 2011
	
\item
\label{ADPZ2}[ADPZ2]
	S. Adams, N. Dirr, M. A. Peletier, and J. Zimmer. 
	Large deviations and gradient flows. 
	http://arxiv.org/abs/1201.4601 , 2012.

\item
\label{AGS}[AGS]
	L. Ambrosio, N. Gigli, and G. Savar\'e,
	Gradient Flows In Metric Spaces And In The Space Of Probability Measures,
	Springer, 2005
	
\item
\label{ASZ}[ASZ]
	L. Ambrosio, G. Savar\'e and L. Zambotti,
	Existence and Stability for Fokker-Planck equations with log-concave reference measure,
	Probability Theory and Related Fields, vol. 145, 517-564 (2009)

\item
\label{BLM}[BLM]
	L. Bertini, C. Landim and M. Mourragui,
	Dynamical large deviations for the boundary driven weakly asymmetric exclusion process. 
	\textit{Ann. Probab.} 37 (2009), no. 6, 2357--2403.

\item
\label{DLR}[DLR]
	Manh Hong Duong, Vaios Laschos, and D.R.M. Renger. 
	Wasserstein gradient flows from large deviations of thermodynamic limits.
	\textit{ESAIM: COCV}, 19(4), 1166-1188, 2013.
	
\item
\label{DG}[DG]
	D. A. Dawson and J. Gartner
	Large deviations from the McKean-Vlasov limit for weakly interacting diffusions.
	\textit{Stochastics} 20 (1987), no. 4, 247--308. 
	
\item
\label{DGMT}[DGMT]
	De Giorgi, Ennio; Marino, Antonio; Tosques, Mario
	Problems of evolution in metric spaces and maximal decreasing curve. 
	\textit{Atti Accad. Naz. Lincei Rend. Cl. Sci. Fis. Mat. Natur.} (8) 68 (1980), no. 3, 180--187. 
	
\item
\label{DPZ}[DPZ]
	M. H. Duong, M. Peletier and J. Zimmer,
	GENERIC formalism of a Vlasov-Fokker-Planck equationand connection to Large Deviation Principle. 
	\textit{Nonlinearity} 26 (2013) 2951-2971. 
	
\item
\label{DV}[DV]
	M. Donsker and S.R.S. Varadhan,
	Large deviations from a hydrodynamic scaling limit.
	\textit{Comm. Pure Appl. Math.}, Number 42 (1989), 243-270.
	
\item
\label{DZ}[DZ]
	A. Dembo and O. Zeitouni,
	Large deviations techniques and applications.
	\textit{Stochastic Modelling and Applied Probability}, 38. Springer-Verlag, Berlin, 2010. xvi+396 pp.
	
\item
\label{F1}[F1]
	M. Fathi
	A two-scale approach to the hydrodynamic limit, part II : local Gibbs behavior.
	\textit{ALEA Lat. Am. J. Probab. Math. Stat.} 10 (2013), no. 2, 625–651.

\item
\label{F2}[F2]
	M. Fathi,
	Modified logarithmic Sobolev inequalities for canonical ensembles.
	preprint, http://arxiv.org/abs/1306.1484
	
\item
\label{FM}[FM]
	M. Fathi and G. Menz,
	Hydrodynamic limit for conservative spin systems with super-quadratic, partially inhomogeneous single-site potential. 
	preprint, http://arxiv.org/abs/1405.3327
	
\item
\label{FK}[FK]
	J. Feng and Thomas G. Kurtz. 
	Large deviations for stochastic processes, 
	volume 131 of Mathematical Surveys and Monographs. 
	American Mathematical Society, Providence, RI, 2006
		
\item
\label{Fo}[Fo]
	Follmer, H.,
	Random fields and diffusion processes,
	Ecole d'Et\'e de Probabilit\'es de Saint-Flour XV-XVII, 1985-87,
	Lecture Notes in Math., 1362, Springer, 1988. (101-203).
	
\item
\label{Fr}[Fr]
	J. Fritz,
	Hydrodynamics in a symmetric random medium.
	\textit{Comm. Math. Phys.}, 125, 13-25, (1989).
	
\item
\label{GOVW}[GOVW]
	N. Grunewald, F. Otto, C. Villani and M. G. Westdickenberg,
	A two-scale approach to logarithmic Sobolev inequalities and the hydrodynamic limit.
	\textit{Ann. Inst. H. Poincar\'e Probab. Statist}. 
	45 (2009), 2, 302--351.

\item
\label{GPV}[GPV]
	M.Z. Guo, G.C. Papanicolaou and S.R.S. Varadhan,
	Nonlinear Diffusion Limit for a System with Nearest Neighbor Interactions,
	\textit{Commun. Math. Phys.} 118, 31-59 (1988)
	
\item
\label{JKO}[JKO]
	R., D. Kinderlehrer, and F. Otto.
	The variational formulation of the Fokker-Planck equation.
	SIAM J. Math. Anal., \textbf{29(1)}:1-17, (1998).
	
\item
\label{K}[K]
	E. Kosygina,
	The Behavior of the Specific Entropy in the Hydrodynamic Scaling Limit for the Ginzburg-Landau Model,
	\textit{Markov Processes and Related Fields}, 
	\textbf{7}, 3 (2001), pp. 383-417.
	
\item
\label{KV}[KV]
	C. Kipnis and S.R.S. Varadhan,
	Central limit theorem for additive functionals of reversible Markov processes and applications to simple exclusions.
	\textit{Commun.Math.Phys.} 104, 1-19 (1986)
	
\item
\label{L}[L]
	S. Lisini,
	Nonlinear diffusion equations with variable coefficients as gradient flows in Wasserstein spaces,
	ESAIM: Control Optimization Calculus of Variations 15 (2009) 712-740
	
\item
\label{Le}[Le]
	C. Leonard,
	Girsanov Theory under a finite entropy condition,
	S\'eminaire de probabilit\'es de Strasbourg, vol. XLIV. 
	Lecture Notes in Mathematics 2046, Springer-Verlag, 2012, 429-465.

\item
\label{M}[M]
	J. Maas,
	Gradient flows of the entropy for finite Markov chains
	\textit{J. Funct. Anal.} 261 (8) (2011), 2250-2292.
	
\item
\label{Ma}[Ma]
	M. Mariani,
	A Gamma-convergence approach to large deviations, 
	\newline http://arxiv.org/abs/1204.0640
	
\item
\label{O}[O]
	F. Otto, 
	The geometry of dissipative evolution equations: the porous medium equation.
	\textit{Comm. Partial Differential Equations}, \textbf{26}(1-2):101-174, (2001).
	
\item
\label{OV}[OV]
	F. Otto and C. Villani, 
	Generalization of an inequality by Talagrand and links with the logarithmic Sobolev inequality.
	\textit{J. Funct. Anal.}, \textbf{173}(2):361-400, (2000).

\item
\label{Q}[Q]
	J. Quastel, 
	Large deviations from a hydrodynamic scaling limit for a nongradient system. 
	Ann. Probab. 23 (1995), no. 2, 724–742. 

\item
\label{S}[S]
	S. Serfaty,
	Gamma-convergence of gradient flows on Hilbert and metric spaces and applications,
	\textit{Disc. Cont. Dyn. Systems}, A, 31, No 4, (2011), 1427-1451

\item
\label{Va}[Va]
	S.R.S. Varadhan,
	Nonlinear diffusion limit for a system with nearest neighbor interactions II,
	in \textit{Asymptotic problems in probability theory : stochastic models and diffusions on fractals},
	Pitman Research Notes in Mathematics Series, \textbf{283}.
	
\item
\label{Vi1}[Vi1]
	C. Villani,
	Topics in Optimal Transportation,
	Graduate Studies in Mathematics, Vol. \textbf{58},
	American Mathematical Society.
	
\item
\label{Vi2}[Vi2]
	C. Villani,
	Optimal Transport, Old and New.
	\textit{Grundlehren der mathematischen Wissenschaften},
	Vol. 338, Springer-Verlag, 2009.
	
\item
\label{Y}[Y]
	H.T. Yau,
	Relative Entropy and Hydrodynamics of Ginzburg-Landau Models, 
	\textit{Lett. Math. Phys.}, \textbf{22} (1991), 63-80.

\end{itemize}

\end{document}